\documentclass{amsart}

\usepackage{amssymb}

\usepackage{amsfonts}
\usepackage{amsmath}
\usepackage{mathtools}
\usepackage{amsthm}

\newtheorem{thm}{Theorem}[section]
\newtheorem{prop}[thm]{Proposition}
\newtheorem{lem}[thm]{Lemma}
\newtheorem{cor}[thm]{Corollary}

\theoremstyle{definition}

\theoremstyle{remark}

\newtheorem{remark}[thm]{Remark}
\newtheorem{hyp}[thm]{Hypothesis}

\numberwithin{equation}{section}

\newcommand{\R}{\mathbf{R}}  

\DeclareMathOperator{\diff}{\mathcal{D}}

\DeclareMathOperator{\lien}{\mathfrak{n}}
\DeclareMathOperator{\lieg}{\mathfrak{g}}
\DeclareMathOperator{\liet}{\mathfrak{t}}
\DeclareMathOperator{\liez}{\mathfrak{z}}
\DeclareMathOperator{\lieb}{\mathfrak{b}}

\DeclareMathOperator{\cO}{\mathcal{O}}
\DeclareMathOperator{\cK}{\mathcal{K}}
\DeclareMathOperator{\cI}{\mathcal{I}}
\DeclareMathOperator{\cG}{\mathcal{G}}
\DeclareMathOperator{\cT}{\mathcal{T}}

\DeclareMathOperator{\stab}{Stab}
\DeclareMathOperator{\lie}{Lie}
\DeclareMathOperator{\spec}{Spec}

\DeclareMathOperator{\bX}{\mathbb{X}}
\DeclareMathOperator{\bC}{\mathbb{C}}
\DeclareMathOperator{\bG}{\mathbb{G}}
\DeclareMathOperator{\bA}{\mathbb{A}}

\makeatletter
\newcommand{\colim@}[2]{%
  \vtop{\m@th\ialign{##\cr
    \hfil$#1\operator@font colim$\hfil\cr
    \noalign{\nointerlineskip\kern1.5\ex@}#2\cr
    \noalign{\nointerlineskip\kern-\ex@}\cr}}%
}
\newcommand{\colim}{%
  \mathop{\mathpalette\colim@{\rightarrowfill@\textstyle}}\nmlimits@
}
\makeatother

\makeatletter
\newcommand{\nlim@}[2]{%
  \vtop{\m@th\ialign{##\cr
    \hfil$#1\operator@font lim$\hfil\cr
    \noalign{\nointerlineskip\kern1.5\ex@}#2\cr
    \noalign{\nointerlineskip\kern-\ex@}\cr}}%
}
\newcommand{\nlim}{%
  \mathop{\mathpalette\nlim@{\leftarrowfill@\textstyle}}\nmlimits@
}
\makeatother

\newcommand{\dashdownarrow}{\raisebox{2.0ex}{\rotatebox{-90}{$\dashrightarrow$}}}

\begin{document}

\title{A Fourier transform for the quantum Toda lattice}

\author{Gus Lonergan}
\address{Department of Mathematics, Massachusetts Institute of Technology, 
Cambridge, MA 02139}
\email{gusl@mit.edu}


\begin{abstract}
We introduce an algebraic Fourier transform for the quantum Toda lattice. 
\end{abstract}

 \maketitle
\setcounter{tocdepth}{1}
\tableofcontents




\section{Introduction}

\subsection{The Toda lattice.}\label{toda}Following \cite{kost}, let $G$ be a complex reductive algebraic group and denote by $Toda_1(G)$ the partially compactified quantum Toda lattice of $G$. By definition, this is the two-sided\footnote{The epithet `two-sided' refers to the use of both the left- and the right-regular actions of $N$ on $G$.} quantum Hamiltonian reduction of $\diff(G)$ with respect to a generic character $\psi$ of a maximal unipotent subgroup $N$. In a formula, we have
$$
Toda_1(G)=(\diff(G)/((l-\psi)(\lien)+(r-\psi)(\lien))\diff(G))^{N\times N}
$$
where $l$, resp. $r$ are the embeddings of $\lieg$ in $\diff(G)$ as left-, resp. right-invariant vector fields. Here $\lieg=\lie(G)$, $\lien=\lie(N)$ (and this pattern will continue). It is naturally a Hopf algebroid over $\cO(\lieg^*//G)$. A different choice of $N,\psi$ gives a canonically isomorphic Hopf algebroid (justifying the definite article). Kostant's classic result gives a canonical isomorphism between $Toda_1(G)$ and the quantum Hamiltonian reduction of $\diff(G)$ with respect to the trivial character of $G$ itself, acting adjointly.

The order filtration on $\diff(G)$ induces\footnote{This is true for the presentation as $\diff(G)//_{triv}G$. For the presentation as $N_\psi\backslash\backslash\diff(G)//_\psi N$ one needs to adjust this filtration by the $\rho^\vee$-weight.} a filtration on $Toda_1(G)$. The base $\cO(\lieg^*//G)$, viewed as a subalgebra of $Toda_1(G)$, is canonically isomorphic to its associated graded and one thus obtains an associated graded Hopf algebroid (over the same base). This is the partially compactified Toda lattice and will be denoted $Toda_0(G)$. In fact, it is a commutative Hopf algebra over $\cO(\lieg^*//G)$, and its corresponding group scheme is canonically identified with the flat abelian group scheme $\liez_{\psi+\lien^\perp}^G/N$ over $(\psi+\lien^\perp)/N\cong \lieg^*//G$ \footnote{In fact, $\liez_{\psi+\lien^\perp}^G$ is itself a flat abelian group scheme over $\psi+\lien^\perp$, on which base $N$ acts freely, justifying the notation `$/N$' rather than `$//N$'. It is customary to trivialize the $N$-torsor $\psi+\lien^\perp\to \lieg^*//G$. The resulting section $\kappa$ is called `the' Kostant slice, and we obtain $\spec(Toda_0(G))\cong \liez_\kappa^G$.}. Here $\lien^\perp$ denotes the orthogonal complement to $\lien$ in $\lieg^*$. In this classical setting, Kostant's result identifies this group scheme with the adjoint quotient $\liez_{\lieg^*}^G//G$.

Let us now fix a maximal torus $T$ of the normalizer $B$ of $N$. We obtain an opposite subgroup $N_-$ to $N$, and a splitting $\lieg^*=\lien^*\times\liet^*\times\lien_-^*$. We can therefore form the subscheme $\psi+\liet^*\subset\psi+\lien^\perp$. This fits into a commutative diagram
$$
\begin{matrix}\psi+\liet^* &\xrightarrow{\sim}&\liet^*\\
\downarrow&&\downarrow\\
\lieg^*//G&\xrightarrow{\sim}&\liet^*//W\end{matrix}
$$   
and consequently we have an isomorphism $\liez_{\psi+\liet^*}^G\cong\liet^*\times_{\lieg^*//G}(\liez_{\psi+\lien^\perp}^G/N)$. In other words, there exists an action of $W$ on $\liez_{\psi+\liet^*}^G$, compatible with its usual action on the base $\liet^*$, and the geometric quotient is the spectrum of the partially compactified Toda lattice. Noting that $\psi+\liet^*$ is contained in $(\lieb^*)^{reg}$, we see that $\liez_{\psi+\liet^*}^G=\liez_{\psi+\liet^*}^B$. The resulting projection $\liez_{\psi+\liet^*}^G\to(\psi+\liet^*)\times T$ is an isomorphism over $\psi+(\liet^*)^{reg}$. This map is $W$-equivariant, where $W$ acts `diagonally' on $(\psi+\liet^*)\times T\cong \liet^*\times T$. All this goes to show that we have a map
\begin{align*}
Spec(Toda_0(G))\to (T^*T)//W
\end{align*}
which is an isomorphism generically over the base $\liet^*//W$.

This classical picture is quantizes in the natural manner, and in particular we have a map
\begin{align}
\diff(T)^W\to Toda_1(G)\label{comparison}
\end{align}
of Hopf algebroids over $\liet^*//W$, which is generically an isomorphism.

\subsection{The Fourier transform.}One is interested in understanding modules for $Toda_1(G)$. Restricting along \ref{comparison}, such a thing becomes a module for $\diff(T)^W$. Since $\diff(T)^W$ is Morita equivalent to $\diff(T)\#\bC W$, this is the same thing as a $W$-equivariant $\diff$-module on $T$. There is a well-known equivalence
\begin{align}
\diff(T)-mod\cong QCoh^{\bX^\bullet(T)}(\liet^*)
\end{align}
and likewise
\begin{align}
\diff(T)^W-mod\cong QCoh^{\widetilde{W}^{aff}}(\liet^*)\label{babyfourier}
\end{align}
where $\bX^\bullet(T)$ denotes the character lattice of $T$ and $\widetilde{W}^{aff}=\bX^\bullet(T)\#W$ is the partially-extended affine Weyl group\footnote{As opposed to the fully extended affine Weyl group, which is usually defined to be the group obtained in this manner starting from the universal cover of $G$. The affine Weyl group, $W^{aff}$, is the group obtained in this manner starting from the adjoint quotient of $G$.}. This equivalence may be regarded as an algebraic incarnation of the Fourier transform, but it is completely trivial when one writes out the definition of the categories to be related.

A natural question arises: is there a similar kind of `algebraic Fourier transform' for $Toda_1(G)$? That is the subject of this paper. In fact, we have
\begin{thm}\label{thm1}There exists an equivalence of categories
\begin{align}
Toda_1(G)-mod\cong QCoh^{\widetilde{W}^{aff}}(\liet^*)^{``\cI"}
\end{align}
which is compatible with \ref{babyfourier} in the natural way.\end{thm}
Here the category $QCoh^{\widetilde{W}^{aff}}(\liet^*)^{``\cI"}$ denotes the full subcategory of $QCoh^{\widetilde{W}^{aff}}(\liet^*)$ whose objects are all those with trivial \emph{derived isotropy} for $W^{aff}$. The precise meaning of this will be spelled out in the main body of the paper (see Proposition \ref{bigprop}). An alternative formulation is as follows:
\begin{thm}\label{thm2?}There exists an equivalence of categories
\begin{align}
Toda_1(G)-mod\cong QCoh^{\widetilde{W}^{aff}}(\liet^*)^{fpd}
\end{align}
which is compatible with \ref{babyfourier} in the natural way.\end{thm}
Here the category $QCoh^{\widetilde{W}^{aff}}(\liet^*)^{fpd}$ denotes the full subcategory of $QCoh^{\widetilde{W}^{aff}}(\liet^*)$ whose objects are all those whose $\Gamma$-equivariant structure descends to $\liet^*//\Gamma$, for every finite parabolic subgroup $\Gamma\subset W^{aff}\subset\widetilde{W}^{aff}$.
\section{The Affine Grassmannian}

For background material on the affine Grassmannian, see \cite{Ginzburg}\cite{Lu}. For algebraic groups, see \cite{Groups}. The material of paragraphs \ref{2.1}-\ref{2.5}, \ref{2.7}, \ref{2.9} is borrowed from these.

\subsection{}\label{2.1}Let $G^\vee$ be the complex algebraic group which is Langlands dual to $G$ and has maximal torus $T^\vee$. Let $Gr$ denote the affine Grassmannian for the $G^\vee$. This is a certain projective ind-scheme whose $\bC$-valued points are $G^\vee(\cK)/G^\vee(\cO)$, where $\cK=\bC((t))$, $\cO=\bC[[t]]$. The translation action of $G^\vee(\cO)$ has finite dimensional orbits, whose closures (the so-called `spherical Schubert varieties') give the ind-scheme structure. The cocharacter lattice $\bX_\bullet(T^\vee)$ embeds in $Gr$ as its $T^\vee$-fixed point subset, and each $G^\vee(\cO)$-orbits contains a unique $W$-orbit in $\bX_\bullet(T^\vee)$. The group $\bG_m$ also acts, by the `loop rotation' local automorphisms of $\cK$, and this action fixes the $T^\vee$-fixed points and preserves the spherical Schubert varieties.

\subsection{}\label{cohom}We will be interested the (complex) cohomology and homology of $Gr$ \footnote{Strictly speaking, we are working here with the ind-analytic space $Gr^{an}$.}. By definition, the cohomology $H^\bullet(Gr)$ of $Gr$ is the cofiltered system of the cohomologies of the spherical Schubert varieties, and the homology $H_\bullet(Gr)$ of $Gr$ is the filtered system of the homologies of the spherical Schubert varieties. The transition morphisms in $H^\bullet(Gr)$ are all surjective, and those in $H_\bullet(Gr)$ are all injective. Forgetting the ind-scheme structure on $Gr$, one obtains the topological space $Gr^{top}$ \footnote{Again, strictly speaking this is the filtered colimit in the category of topological spaces of the filtered system of analytic spaces $Gr^{an}$.}. We then have $H_\bullet(Gr^{top})=\colim H_\bullet(Gr)$. However, $\nlim H^\bullet(Gr)$ is some completion of $H^\bullet(Gr^{top})$ (and so bigger than it). Since each spherical Schubert variety is compact, its homology is equal to its Borel-Moore homology, i.e. its cohomology with coefficients in its dualizing complex. Gluing together these dualizing complexes by the $!$-pullbacks, one may think of $H_\bullet(Gr)$ rather literally as its cohomology with coefficients in \emph{its} dualizing complex.

\subsection{}\label{2.3}We fix a Borel subgroup $B^\vee$ of $G^\vee$ containing $T^\vee$ and write $I^\vee=G^\vee(\cO)\times_{G^\vee}B^\vee$ for the corresponding Iwahori subgroup of $G^\vee(\cO)$. The orbits of $I^\vee$ on $Gr$ form a complex cell decomposition, compatible with the stratification by spherical Schubert varieties. It follows that $Gr$ is equivariantly formal (for, say the $T^\vee$-action). We then define $H^\bullet_{T^\vee}(Gr)$ as the cofiltered system of the $T^\vee$-equivariant cohomologies of the spherical Schubert varieties, and $H_\bullet^{T^\vee}(Gr)$ as the filtered system of the $T^\vee$-equivariant cohomologies of the spherical Schubert varieties \emph{with coefficients in} their respective (canonically equivariant) dualizing complexes\footnote{As before, we like to think of this as the equivariant cohomology of $Gr$ with coefficients in its dualizing complex.}. By equivariant formality (and finite-dimensionality of the spherical Schubert varieties), $H_\bullet^{T^\vee}(Gr)$ and $H^\bullet_{T^\vee}(Gr)$ are dual over $H^\bullet_{T^\vee}(pt)=\cO(\liet^*)$. Similarly we define the $G^\vee$-equivariant homology and cohomology; it is a standard consequence of equivariant formality that
$$
H^\bullet_{T^\vee}(Gr)=\pi^*H^\bullet_{G^\vee}(Gr)
$$
and
$$
H_\bullet^{T^\vee}(Gr)=\pi^*H_\bullet^{G^\vee}(Gr)
$$
where $\pi: \liet^*\to\liet^*//W=\spec(H^\bullet_{G^\vee}(pt))$ is the natural projection. Inversely, $W$ acts naturally on $H^\bullet_{T^\vee}(Gr)$ (resp. $H_\bullet^{T^\vee}(Gr)$), and one obtains $\pi^*H^\bullet_{G^\vee}(Gr)$ (resp. $\pi^*H_\bullet^{G^\vee}(Gr)$) from them by taking invariants. All of the above goes through when one introduces the additional `loop rotation' factor of equivariance (and the corresponding additional $\bA^1$ factor in the base). As in the non-equivariant case, the transition maps in cohomology are all surjective, while those in homology are all injective, and the analogous remarks comparing their (co-) limits to the equivariant (co-) cohomology of $Gr^{top}$ hold.

\subsection{}\label{2.4}Since $G^\vee$ is the maximal reductive quotient of $G^\vee(\cO)$, we may replace $G^\vee$ by $G^\vee(\cO)$ as a factor of equivariance. That is to say, we have natural isomorphisms
$$
H^\bullet_{G^\vee}(Gr)\cong H^\bullet_{G^\vee(\cO)}(Gr)
$$
and
$$
H^\bullet_{G^\vee\times\bG_m}(Gr)\cong H^\bullet_{G^\vee(\cO)\rtimes\bG_m}(Gr)
$$
where the right-hand sides are defined just as before (recall that the spherical Schubert varieties are $G^\vee(\cO)$-equivariant by definition). We have the analogous isomorphisms in homology. In the same fashion, one may replace $T^\vee$ by $I^\vee$.

\subsection{}\label{2.5}Putting the ind-scheme structure of $Gr$ together with the pro-scheme structure of $G^\vee(\cO)$, we may regard $G^\vee(\cK)$ as a pro-ind-scheme. With some rearranging, it may be viewed as an ind-pro-scheme, admitting the $G^\vee(\cO)\times G^\vee(\cO)$-orbits as a final family of sub-pro-schemes. We have
$$
H^\bullet_{G^\vee(\cO)}(Gr)=H^\bullet_{G^\vee(\cO)\times G^\vee(\cO)}(G(\cK))
$$
$$
H^\bullet_{G^\vee(\cO)\rtimes\bG_m}(Gr)=H^\bullet_{(G^\vee(\cO)\times G^\vee(\cO))\rtimes\bG_m}(G(\cK))
$$
and so on. This symmetric presentation makes it clear that in both cases we have an extra action of $\cO(\liet^*//W)$. It happens that the two $\cO(\liet^*//W)$-module structures differ in the presence of loop-rotation equivariance, and coincide without it.

\subsection{A remark on equivariant homology.}\label{2.6}The description of $H_\bullet^{G^\vee(\cO)\rtimes\bG_m}(Gr)$ as the $\cO(\liet^*//W\times\bA^1)$-linear dual to $H^\bullet_{G^\vee(\cO)\rtimes\bG_m}(Gr)=H^\bullet_{(G^\vee(\cO)\times G^\vee(\cO))\rtimes\bG_m}(G(\cK))$ (using the `left-hand' action of $\cO(\liet^*//W)$) appears asymmetrical. However, essentially because $G^\vee(\cO)$ is (pro-) smooth, the $*$-pullback to $G^\vee(\cK)$ of the dualizing complex on $Gr$ descends to the dualizing complex on $G^\vee(\cO)\backslash G^\vee(\cK)$. One can then give a symmetric description of $H_\bullet^{G^\vee(\cO)\rtimes\bG_m}(Gr)$ as the $(G^\vee(\cO)\times G^\vee(\cO))\rtimes\bG_m$-equivariant cohomology of $G^\vee(\cK)$ with coefficients in this pullback\footnote{which is to be thought of as \emph{its} dualizing complex, homologically shifted by $-2\dim G^\vee(\cO)$.}. In particular this gives a canonical isomorphism between the two  $\cO(\liet^*//W\times\bA^1)$-linear duals to $H^\bullet_{G^\vee(\cO)\rtimes\bG_m}(Gr)$ \footnote{If one restricts to $0\in\bA^1$, these two duals are identical and the resulting automorphism is the identity.}.

This line of reasoning has an interesting application for finite flag varieties. Namely, let $P$ and $Q$ be two parabolic subgroups of $G$. Since $P,Q$ are smooth and $P\backslash G,G/Q$ are compact, the $P\times Q$-equivariant cohomology of the dualizing complex on $G$ is at once the $H^\bullet_P(pt)=\cO(\liet//W_P)$-linear dual of $H^\bullet_P(G/Q)$, shifted by $2\dim Q$, and the $H^\bullet_Q(pt)=\cO(\liet//W_Q)$-linear dual of $H^\bullet_Q(P\backslash G)$, shifted by $2\dim P$. The consequence in algebraic geometry is that, whenever $P\subset Q$, we have the canonical isomorphism of functors
$$
(\pi^P_Q)^!(2\dim P)\cong (\pi^P_Q)^*(2\dim Q):QCoh^{\bG_m}(\liet//W_Q)\to QCoh^{\bG_m}(\liet//W_P)
$$
exactly mirroring the topological statement that
$$
(\pi^P_Q)^![2\dim P]\cong (\pi^P_Q)^*[2\dim Q]:D^b(BQ)\to D^b(BP).
$$
Here $\pi^P_Q$ stands simultaneously (and very abusively) for the natural maps $\liet//W_P\to\liet//W_Q$ and $BP\to BQ$.

\subsection{Convolution.}\label{2.7}The multiplication and inversion in the group $G^\vee(\cK)$ are bounded with respect to the ind-pro-structure. It follows that $H^\bullet_{G^\vee(\cO)\rtimes\bG_m}(Gr)$ and $H_\bullet^{G^\vee(\cO)\rtimes\bG_m}(Gr)$ are Hopf algebroids over $\liet^*//W\times\bA^1$. They are supported on the diagonal $\liet^*//W\times\liet^*//W\times\bA^1\subset (\liet^*//W\times\bA^1)^2$, and thus we may think of them as an $\bA^1$-family of Hopf algebroids over $\liet^*//W$. The same is true for their (co-) limits discussed in \ref{cohom}. Of course some care has to be taken to say what it means for a pro- or ind-object to be a Hopf algebroid. For our purposes, it happens that the pro-object $H^\bullet_{G^\vee(\cO)\rtimes\bG_m}(Gr)$ is a pro-algebra (but not a pro-coalgebra), while the ind-object $H_\bullet^{G^\vee(\cO)\rtimes\bG_m}(Gr)$ is an ind-coalgebra (but not an ind-algebra). We do not know whether these restrictions are critical to have a good general theory, but they are very natural in the topological setting.

The pro-object Hopf algebroid $H^\bullet_{G^\vee(\cO)\rtimes\bG_m}(Gr)$ is commutative. Equivalently, its spectrum is a groupoid ind-scheme. In a sense which will be made precise, the ($\bA^1$-families of) representations of this groupoid over $\liet^*//W$ are the same as comodules for $H^\bullet_{G^\vee(\cO)\rtimes\bG_m}(Gr)$, which in turn (by pro-finite flatness) are the same as modules for the dual algebra $H_\bullet^{G^\vee(\cO)\rtimes\bG_m}(Gr)$.

\subsection{}\label{2.8} Following \cite{BF} (and, originally, \cite{Ginzburg}), the following is a consequence of the geometric Satake equivalence (whose details we need not recall here):
\begin{thm}\label{bez} There are natural isomorphisms
$$
H_\bullet^{G^\vee\rtimes\bG_m}(Gr^{top})\cong Toda_{\hbar}(G)
$$
and
$$
H^\bullet_{G^\vee\rtimes\bG_m}(Gr^{top})\cong \cO(\pi_1(G^\vee)\times N_{\liet^*//W}(\liet^*//W)^2)
$$
of Hopf algebroids over $\liet^*//W\times\bA^1$.
\end{thm}
Here $Toda_{\hbar}(G)$ is the Rees construction of $Toda_1(G)$ with respect to the filtration of \ref{toda}, and $\hbar$ is the parameter of $\bA^1$. Also $N_XY$ denotes the `deformation to the normal cone' of $X$ in $Y$ whenever $X\subset Y$ is a closed subscheme. It is a flat $\bG_m$-equivariant $\bA^1$-scheme whose restriction to $\bA^1-\{0\}$ is $Y\times(\bA^1-\{0\})$ and whose $0$-fiber is the normal cone of $X$ in $Y$. The groupoid scheme structure comes from the trivial groupoid structure on $(\liet^*//W)^2$ (with $\liet^*//W$ its subgroup).\

Recall that $H_\bullet^{G^\vee\rtimes\bG_m}(Gr^{top})=\colim H_\bullet^{G^\vee\rtimes\bG_m}(Gr)$; thus the data provided by $H_\bullet^{G^\vee\rtimes\bG_m}(Gr)$ is nothing more than a filtration of $Toda_\hbar(G)$, the so called spherical Schubert filtration. A module for $Toda_\hbar(G)$ is therefore precisely a trivially filtered module for $H_\bullet^{G^\vee\rtimes\bG_m}(Gr)$. These are equivalent to trivially cofiltered comodules for $H^\bullet_{G^\vee\rtimes\bG_m}(Gr)$, and one obtains:
\begin{cor}There is an equivalence of categories
$$
Toda_\hbar(G)-mod\cong Rep_{\liet^*//W\times\bA^1}\spec(H^\bullet_{G^\vee\rtimes\bG_m}(Gr))
$$
whence
$$
Toda_1(G)-mod\cong Rep_{\liet^*//W}\spec(H^\bullet_{G^\vee\rtimes\bG_m}(Gr))_{\hbar=1}.
$$
\end{cor}
Here $\spec(H^\bullet_{G^\vee\rtimes\bG_m}(Gr))$ is a certain sub-groupoid-ind-scheme of $N_{\liet^*//W}(\liet^*//W)^2$, which the remainder of this paper is aimed at understanding.

\subsection{The affine flag variety.}\label{2.9}The affine flag variety, denoted $Fl$, is a certain $G^\vee/B^\vee$-bundle over $Gr$ whose $\bC$-points are $G^\vee({\cK})/I^\vee$. We define the $\bA^1$-families of Hopf algebroids over $\liet^*$:
$$
H^\bullet_{I^\vee\rtimes\bG_m}(Fl)
$$
$$
H_\bullet^{I^\vee\rtimes\bG_m}(Fl)
$$
in essentially the same way as for $Gr$ \footnote{And as for $Gr$ they may also be regarded as the $(I^\vee\times I^\vee)\rtimes\bG_m$-equivariant cohomology of some `complexes' (the constant sheaf and the shifted dualizing complex) on $G^\vee(\cK)$.}. They are respectively pro-finite flat and ind-finite flat objects with respect to both $\cO(\liet^*\times\bA^1)$-module structures, over which they are also dual. Again $H^\bullet_{I^\vee\rtimes\bG_m}(Fl)$ is commutative, and so its spectrum is an $\bA^1$-family of groupoid ind-schemes over $\liet^*$. It is related to $\spec(H^\bullet_{G^\vee\rtimes\bG_m}(Gr))$ in the following way:
$$
\spec(H^\bullet_{I^\vee\rtimes\bG_m}(Fl))\cong \liet^*\times_{\liet^*//W}\spec(H^\bullet_{G^\vee\rtimes\bG_m}(Gr))\times_{\liet^*//W}\liet^*,
$$
i.e. it is obtained from $\spec(H^\bullet_{G^\vee\rtimes\bG_m}(Gr))$ by applying the natural pullback functor for ind-schemes along $\pi:\liet^*\to\liet^*//W$.

We will see that the representations of these two groupoids are equivalent. Let $\cT$ denote the pullback groupoid evaluated at $\hbar=1$. It is the study of this groupoid which will eventually yield Theorem \ref{thm1}.

\subsection{Localization.}\label{2.10}The localization theorem of \cite{GKM} allows us to describe $\cT$ precisely. Indeed they show that the restriction map
$$
H^\bullet_{I^\vee\rtimes\bG_m}(Fl)\cong H^\bullet_{T^\vee\rtimes\bG_m}(Fl)\to H^\bullet_{T^\vee\rtimes\bG_m}(Fl^{T^\vee\times\bG_m})
$$
is generically\footnote{This means that for each closure of $G^\vee(\cO)$-orbit in $Fl$, whose ($T^\vee\times\bG_m$-) equivariant cohomologies form the pieces of the pro-object $H^\bullet_{T^\vee\rtimes\bG_m}(Fl)$, the restriction map from its equivariant cohomology to that of its ($T^\vee\times\bG_m$-) fixed point set is an isomorphism over some non-empty Zariski-open subset of the base $\liet^*\times\bA^1$. These open subsets do not stabilize as one exhausts $Fl$.} an isomorphism (and so in particular injective, since the left hand side is torsion-free over $\liet^*\times\bA^1$). The same holds when we set $\hbar=1$, by gradedness. The $T^\vee\times\bG_m$-fixed point set of $Fl$ is identified with $\widetilde{W}^{aff}$, and the spectrum of its equivariant cohomology is the ind-scheme
$$
\coprod_{\gamma\in \widetilde{W}^{aff}}\Gamma^\hbar_\gamma.
$$
Here $\Gamma^\hbar_\gamma=\{(x,y,\hbar\in \liet^*\times\liet^*\times\bA^1)|x=\gamma^\hbar(y)\}$ and $\gamma^\hbar$ denotes the operator obtained from $\gamma$ by dilating its translational part by $\hbar$. This is the same as the \emph{transformation groupoid}
$$
\widetilde{W}^{aff}\# (\liet^*\times\bA^1)
$$
where $\widetilde{W}^{aff}$ now is regarded as a group ind-scheme which acts on $\liet^*\times\bA^1$ by $\gamma\mapsto\gamma^\hbar$. In particular setting $\hbar=1$ we have
$$
\spec(H^\bullet_{T^\vee\rtimes\bG_m}(Fl^{T^\vee\times\bG_m}))_{\hbar=1}\cong \widetilde{W}^{aff}\#\liet^*.
$$
Representations over $\liet^*$ of this groupoid are by definition precisely $\widetilde{W}^{aff}$-equivariant quasicoherent sheaves on $\liet^*$. Identifying $H_\bullet^{T^\vee\rtimes\bG_m}(Fl^{T^\vee\times\bG_m})_{\hbar=1}$ with $\diff(T)\#\bC W$ gives the algebraic Fourier transform of \ref{babyfourier}. Therefore it is natural to regard the category of representations of $\cT$ as the `Fourier transform description' of the category of modules for $Toda_1(G)$.

Since 
$$
H^\bullet_{T^\vee\rtimes\bG_m}(Fl)_{\hbar=1}\to H^\bullet_{T^\vee\rtimes\bG_m}(Fl^{T^\vee\times\bG_m})_{\hbar=1}
$$
is injective, and
$$
\cO(\pi_1(G^\vee)\times\liet^*\times\liet^*)\cong H^\bullet_{T^\vee\rtimes\bG_m}(Fl^{top})_{\hbar=1}\to H^\bullet_{T^\vee\rtimes\bG_m}(Fl)_{\hbar=1}
$$
is surjective in every piece of the right-hand pro-object, we obtain:
\begin{lem}\label{imaget}$\cT$ is the image of the morphism
$$
\widetilde{W}^{aff}\#\liet^* \xrightarrow{(p\circ\pi_1,i)} \pi_1(G^\vee)\times\liet^*\times\liet^*.
$$\end{lem}
Here $p:\widetilde{W}^{aff}\to \pi_1(G^\vee)$ is the quotient map (with kernel $W^{aff}$) and $i:\widetilde{W}^{aff}\#\liet^*\to \liet^*\times\liet^*$ is the closed embedding.

The idea now is that this image is in some reasonable sense the quotient of $\widetilde{W}^{aff}\#\liet^*$ by the isotropy (i.e. maximal) subgroup of $W^{aff}\#\liet^*$ - as for instance can be seen on closed points - and thus one expects that representations of $\cT$ are $\widetilde{W}^{aff}$-equivariant quasicoherent sheaves on $\liet^*$ with trivial isotropy in $W^{aff}$. However, the fact that the isotropy subgroup is not flat over $\liet^*$ causes difficulties, and in fact this is why we must introduce the notion of `derived isotropy'. We deal with these issues in the following section.

We make one final remark. The localization theorem of \cite{GKM} also includes a description of the image of $H^\bullet_{T^\vee\rtimes\bG_m}(Fl)\to H^\bullet_{T^\vee\rtimes\bG_m}(Fl^{T^\vee\times\bG_m})$. We have of course already described this (via the calculation of $H^\bullet_{T^\vee\rtimes\bG_m}(Fl^{top})$), but the result of \cite{GKM} provides some crucial and not obvious additional information (see Section \ref{checktion}).

\section{Groupoids and descent}

\subsection{}This section contains a number of general lemmas on groupoid schemes which would suffice to prove Theorem \ref{thm1} if $\cT$ were a scheme rather than an ind-scheme. In fact, the arguments carry over more or less directly to the case of ind-schemes, but for readability we have chosen to leave the thorough treatment of ind-schemes to the following section.

Let $\mathcal{G}\rightrightarrows X$ be an algebraic groupoid over the scheme $X$. We denote by $s$, $t$ the two maps to $X$ (heads and tails). Expressions $\mathcal{G}\times_X$, $\times_X\mathcal{G}$ denote the Cartesian product using respectively $s$, $t$, and likewise expressions $\mathcal{O}(\mathcal{G})\otimes_{\mathcal{O}(X)}$, $\otimes_{\mathcal{O}(X)}\mathcal{O}(\mathcal{G})$ denote the tensor product using respectively $s^\#$, $t^\#$.

\subsection{A lemma on descent.}We make the important assumption that $t$ is flat. In that case, the category $Rep_X(\mathcal{G})$ of right $\mathcal{O}(\mathcal{G})$-comodules is abelian. We have:

\begin{lem}[faithfully flat descent]\label{lemma1}Suppose $f:Y\to X$ is faithfully flat (or more generally that it induces a universally injective map $\mathcal{O}(X)\to\mathcal{O}(Y)$ of $\mathcal{O}(X)$-modules). Then the functor \begin{align} f^*:Rep_X(\mathcal{G})\to Rep_Y(Y\times_X\mathcal{G}\times_XY)\end{align} is an equivalence of (monoidal) categories. \end{lem}
\begin{remark}The usual statement of faithfully flat descent is the special case where $\mathcal{G}$ is the trivial groupoid.\end{remark}
\begin{proof}This is an instructive application of Barr-Beck. Consider the composition\begin{align}Rep_X(\mathcal{G})\xrightarrow{Res^\mathcal{G}_X}QCoh(X)\xrightarrow{f^*}QCoh(Y).\end{align}Then \begin{itemize}\item $f^*\circ Res^\mathcal{G}_X$ admits a right adjoint, namely $Coind^\mathcal{G}_X\circ f_*$, where $Coind^\mathcal{G}_X = (-)\otimes_{\mathcal{O}(X)}\mathcal{O}(\mathcal{G})$ is the right adjoint to $Res^\mathcal{G}_X$.\
\item $f^*\circ Res^\mathcal{G}_X$ reflects isomorphisms, since $f^*$ does ($f$ being faithfully flat), and $Res^\mathcal{G}_X$ does (being an exact faithful functor between abelian categories).\
\item $Rep_X(\mathcal{G})$ has, and $f^*\circ Res^\mathcal{G}_X$ preserves, equalizers of $f^*\circ Res^\mathcal{G}_X$-split equalizers. This is because $f^*$ has this property (by hypothesis) and $Res^\mathcal{G}_X$ has the stronger property that $Rep_X(\mathcal{G})$ has, and $Res^\mathcal{G}_X$ preserves, equalizers of $Res^\mathcal{G}_X$-equalizers (being an exact functor between abelian categories).
\end{itemize}Consequently $f^*\circ Res^\mathcal{G}_X$ induces an equivalence between $Rep_X(\mathcal{G})$ and the category of $f^*\circ Res^\mathcal{G}_X\circ Coind^\mathcal{G}_X\circ f_*$-comodules in $QCoh(Y)$. The comonad in question is nothing more than $(-)\otimes_{\mathcal{O}(X)}\mathcal{O}(\mathcal{G})\otimes_{\mathcal{O}(X)}\mathcal{O}(Y)\cong (-)\otimes_{\mathcal{O}(Y)}\mathcal{O}(Y\times_X\mathcal{G}\times_XY)$ with the comonad structure given by the groupoid structure on $Y\times_X\mathcal{G}\times_XY$.\end{proof}

\subsection{Adjacency and isotropy.}Let $C$ be the coequalizer (in affine schemes) of $\mathcal{G}\rightrightarrows X$; that is, $\mathcal{O}(C)=\{x\in\mathcal{O}(X)|s^\#(x)=t^\#(x)\}$ (a subalgebra of $\mathcal{O}(X)$). The \emph{adjacency groupoid} $\mathcal{T}$ is defined by setting $\mathcal{O}(\mathcal{T})$ to be the subalgebra of $\mathcal{O}(\mathcal{G})$ generated by the images of $s^\#$, $t^\#$. This is naturally an algebraic groupoid over $X$ and we have the maps $\mathcal{G}\to\mathcal{T}\to X\times_CX$ of algebraic groupoids over $X$. We also have the \emph{isotropy subgroup} $\mathcal{I}:=\mathcal{G}\times_\mathcal{T}X$ of $\mathcal{G}$ (here $X\to\mathcal{T}$ is the identity section). Let us write $M$ for the kernel of the surjective map $\mathcal{O}(\mathcal{G})\to\mathcal{O}(\mathcal{I})$; $M$ is nothing more than the ideal generated by $(s^\#-t^\#)(\mathcal{O}(X))$.

Heuristically one thinks of $\mathcal{G}$ as a space of arrows with heads and tails in $X$, which are composable in the natural way (and satisfy the groupoid axioms); the relation of being connected by an arrow is an equivalence relation on $X$. Then $\mathcal{T}$ is the adjacency groupoid of the equivalence relation, and $\mathcal{I}$ is the subgroup of $\mathcal{G}$ consisting of all arrows whose head and tail coincide. In some cases $C$ is the space of equivalence classes: for instance in the case of a transformation groupoid $\mathcal{G}=X\times\Gamma$, $C$ is the GIT quotient and if it is also a geometric quotient then this condition is satisfied. In that case, one might hope \footnote{For instance, this holds when $X$ is the reflection representation of a finite complex reflection group $\Gamma$ and $\cG=\Gamma\# X$; see \ref{fpd} to deduce a proof.} that $\mathcal{T}\to X\times_CX$ is an isomorphism. Naively one expects that formulas such as `$\mathcal{T}=\mathcal{G}/\mathcal{I}$' and `$Rep_X(\mathcal{T})=Rep_X(\mathcal{G})^\mathcal{I}$' to hold. Here $Rep_X(\mathcal{G})^\mathcal{I}$ denotes the full subcategory of $Rep_X(\mathcal{G})$ consisting of objects with trivial $\mathcal{I}$-action. However, this simply isn't true in general. Nonetheless we shall demonstrate some appropriate replacement, given the following hypothesis:

\begin{hyp}\label{hyp}\
\begin{enumerate}\item We assume that $\mathcal{G}$, $\mathcal{T}$ are both flat over $X$, with respect to both the heads and tails maps.\
\item We assume that $\mathcal{G}\to\mathcal{T}$ is the coequalizer of $\mathcal{G}\times_\mathcal{T}\mathcal{G}=\mathcal{G}\times_X\mathcal{I}\rightrightarrows\mathcal{G}$. In light of the first hypothesis, on the level of functions, this is the statement that the images of $s^\#$, $t^\#$ in $\mathcal{O}(\mathcal{G})$ together generate the entire subalgebra $\{x\in\mathcal{O}(\mathcal{G})|\Delta(x)-x\otimes 1\in \mathcal{O}(\mathcal{G})\otimes_{\mathcal{O}(X)}M \}$.\end{enumerate}\end{hyp}

\subsection{Behaviour on flat objects.}Consider the functor of restriction $F:Rep_X(\mathcal{T})\to Rep_X(\mathcal{G})^\mathcal{I}$. \begin{lem}\label{lemma2}$F$ induces an equivalence between the full subcategories of objects which are flat as quasicoherent sheaves on $X$.\end{lem}

\begin{proof}To avoid excessive notation, the symbol $\otimes$ will denote $\otimes_{\mathcal{O}(X)}$ unless given some other subscript.

The main point is the essential surjectivity. Let $V$ be an object in the target category. By definition, this is a flat $\mathcal{O}(X)$-module $V$ together with an $\mathcal{O}(\mathcal{G})$-comodule structure \begin{align}m:V\to V\otimes \mathcal{O}(\mathcal{G})\end{align}such that the composition \begin{align}V\xrightarrow{m} V\otimes \mathcal{O}(\mathcal{G})\to V\otimes \mathcal{O}(\mathcal{I})\end{align}coincides with $id\otimes 1$. Consider the diagram\begin{align}\begin{matrix}V &  \xrightarrow{?} & V\otimes\mathcal{O}(\mathcal{G}) & \xrightarrow{p} & V\otimes\mathcal{O}(\mathcal{I})\\
 & & \downarrow{*} & & \downarrow{_{m\otimes id}}\\
&  & V\otimes\mathcal{O}(\mathcal{G})\otimes\mathcal{O}(\mathcal{G}) & \xrightarrow{q} & V\otimes\mathcal{O}(\mathcal{G})\otimes\mathcal{O}(\mathcal{I})\end{matrix}\end{align}where the horizontal maps labelled $p,q$ are the obvious quotient maps. We perform the following diagram chase: $q\circ (id\otimes\Delta)\circ m=q\circ (m\otimes id)\circ m=(m\otimes id)\circ p\circ m= (m\otimes id)\circ p\circ (id\otimes 1) = q\circ (m\otimes id)\circ (id\otimes 1)=q\circ (id\otimes id\otimes 1)\circ m$. Thus $m$ lands inside the equalizer of $V\otimes\mathcal{O}(\mathcal{G})\rightrightarrows V\otimes\mathcal{O}(\mathcal{G})\otimes\mathcal{O}(\mathcal{I})$, the parallel morphisms being $q\circ (id\otimes \Delta)$ and $q\circ (id\otimes id\otimes 1)$. This equalizer is the kernel of \begin{align}V\otimes\mathcal{O}(\mathcal{G})\xrightarrow{q\circ(id\otimes(\Delta-id\otimes 1))}V\otimes\mathcal{O}(\mathcal{G})\otimes\mathcal{O}(\mathcal{I}),\end{align}which since $V$ is flat is precisely $V\otimes \mathcal{O}(\mathcal{T})$.

Thus we have the unique factorization \begin{align}V\to V\otimes\mathcal{O}(\mathcal{T})\to V\otimes\mathcal{O}(\mathcal{G})\end{align}of $m$. That this is an $\mathcal{O}(\mathcal{T})$-comodule structure on $V$ follows from the fact that $V\otimes \mathcal{O}(\mathcal{T})\otimes \mathcal{O}(\mathcal{T})\to V\otimes \mathcal{O}(\mathcal{G})\otimes \mathcal{O}(\mathcal{G})$ is injective. This is because $\mathcal{O}(\mathcal{T})\to\mathcal{O}(\mathcal{G})$ is injective and every $\mathcal{O}(X)$-module in sight is flat.

Next, $F$ is of course faithful, being the identity on underlying $\mathcal{O}(X)$-modules. Finally, $F$ is full on its flat $\mathcal{O}(X)$-modules. Indeed if $V$, $V'$ are any two such, and $V\to V'$ is a morphism of $\mathcal{O}(\mathcal{G})$-comodules then consider the diagram\begin{align}\begin{matrix}
V & \to & V\otimes \mathcal{O}(\mathcal{T}) & \to & V\otimes \mathcal{O}(\mathcal{G})\\ 
\downarrow &  & \downarrow &  & \downarrow\\ 
V' & \to & V'\otimes \mathcal{O}(\mathcal{T}) & \to & V'\otimes \mathcal{O}(\mathcal{G}).
\end{matrix}\end{align}The outer and rightmost squares are both commutative, and by flatness of $V'$ the lower horizontal arrow in the rightmost square is injective. It follows that the leftmost square is commutative, as required.

A closing remark: the proof shows that $F:Hom(V,V')\to Hom(FV,FV')$ is an isomorphism as long as $V'$ is a flat $\mathcal{O}(X)$-module ($V$ may be arbitrary).\end{proof}

\subsection{Fullness.} In fact, $F$ is full under some additional hypotheses which we will describe. First, we will require:

 \begin{hyp}\label{hyp2}$\mathcal{T}, \mathcal{G}$ are finite over $X$ (with respect to both $s$ and $t$).\end{hyp}

It follows that for any objects $V,V'$ in $Rep_X(\mathcal{T})$, the space $Hom_{\mathcal{O}(X)}(V,V')$ is also an object of $Rep_X(\mathcal{T})$, and $Hom_\mathcal{T}(V,V')$ coincides with its maximal invariant submodule:\begin{align}Hom_\mathcal{T}(V,V')= Hom_{\mathcal{O}(X)}(V,V')^\mathcal{T} & = \{f\in Hom_{\mathcal{O}(X)}(V,V')|m(f) = f\otimes 1\}\label{eq9}\\
& = m^{-1}(Hom_{\mathcal{O}(X)}(V,V')\otimes 1)\label{eq10}.\end{align}(Likewise for $\mathcal{G}$). To see this, consider the following diagram:\begin{align}\begin{matrix}V & \xrightarrow{m} & V\otimes\mathcal{O}(\mathcal{T}) & & & & & &\\
\downarrow{f} & & \downarrow{_{f\otimes id}} & & & & & &\\
V' & \xrightarrow{m} & V'\otimes\mathcal{O}(\mathcal{T}) & \xrightarrow{m\otimes id} & V'\otimes\mathcal{O}(\mathcal{T})\otimes\mathcal{O}(\mathcal{T}) & \xrightarrow{id\otimes id\otimes S} & \xrightarrow{id\otimes mult} & V'\otimes\mathcal{O}(\mathcal{T}).\end{matrix}\end{align}The high road is a map of $\mathcal{O}(X)$-modules from $V$ to $V'\otimes_{\mathcal{O}(X)}\mathcal{O}(\mathcal{T})$ with its $\mathcal{O}(X)$-structure coming from $V'$ (or equivalently $t^\#$); this is different from the one we have been using until now. By the finiteness hypothesis this is the same as an element of $Hom_{\mathcal{O}(X)}(V,V')\otimes_{\mathcal{O}(X)}\mathcal{O}(\mathcal{T})$. One checks that the map so constructed from $Hom_{\mathcal{O}(X)}(V,V')$ to $Hom_{\mathcal{O}(X)}(V,V')\otimes_{\mathcal{O}(X)}\mathcal{O}(\mathcal{T})$ is $\mathcal{O}(X)$-linear (in the usual sense) and makes $Hom_{\mathcal{O}(X)}(V,V')$ into a representation of $\mathcal{T}$. The condition that $f$ is invariant is the condition that the high road of the diagram is equal to $f\otimes 1$. This is equal to the low road. Since the tail of the diagram is an isomorphism (in fact, an involution!) it is equivalent to the condition that the square is commutative.

Consequently, $F$ is full if and only if $F$ `reflects invariants': the natural map $V^\mathcal{T}\to V^\mathcal{G}$ is an isomorphism (for all $V$).

\subsection{Reflection of invariants.}We will now give some conditions which guarantee that $F$ reflects invariants independently of any earlier hypothesis (of course finite flatness is required to deduce fullness from this). For instance, one such condition is:

 \begin{hyp}\label{hyp3}\
The equalizer $\mathcal{O}(C)\to\mathcal{O}(X)\rightrightarrows\mathcal{O}(\mathcal{G})$ is split, with a ($t^\#$) $\mathcal{O}(X)$-linear section $\beta$ of $t^\#$.\end{hyp}

Indeed, in that case one may consider the composition $\alpha$\begin{align}V\xrightarrow{m}V\otimes\mathcal{O}(\mathcal{T})\xrightarrow{id\otimes f} V\otimes\mathcal{O}(\mathcal{G})\xrightarrow{id\otimes\beta}V.\end{align}Note that the condition that $\beta$ be ($t^\#$) $\mathcal{O}(X)$-linear is necessary for this to be well-defined. Certainly $\alpha$ is the identity on $V^\mathcal{G}$, and so in particular on $V^\mathcal{T}$. We claim moreover that $\alpha$ is a section of the inclusion $V^\mathcal{T}\to V$; we have \begin{align}m\circ\alpha & =m\circ(id\otimes\beta)\circ(id\otimes f)\circ m\\
&=(id\otimes id\otimes\beta)\circ(m\otimes id)\circ(id\otimes f)\circ m\\
&=(id\otimes id\otimes\beta)\circ(id\otimes id\otimes f)\circ(m\otimes id)\circ m\\
&=(id\otimes id\otimes\beta)\circ(id\otimes id\otimes f)\circ(id\otimes\Delta)\circ m.\end{align}Since $\mathcal{O}(\mathcal{T})$ is generated by $s^\#(\mathcal{O}(X))$ and $t^\#(\mathcal{O}(X))$, we have $m(V)\subset V\otimes s^\#(\mathcal{O}(X))\subset V\otimes \mathcal{O}(\mathcal{T})$. By linearity of $\Delta$, we have therefore $(id\otimes\Delta)\circ m(V)\subset V\otimes 1\otimes s^\#(\mathcal{O}(X))$. Since $\beta\circ s^\#(\mathcal{O}(X))\subset\mathcal{O}(C)$, we get finally $m\circ\alpha(V)\subset V\otimes \mathcal{O}(C)=V\otimes 1\subset V\otimes\mathcal{O}(\mathcal{T})$. We conclude as follows: suppose $v\in V^\mathcal{G}$. Then $v=\alpha(v)\in V^\mathcal{T}$, as required.

It may happen that $\mathcal{O}(X)\rightrightarrows\mathcal{O}(\mathcal{G})$ does not admit a split equalizer globally, but does locally. Thus we make:

\begin{hyp}[Replacement for Hypothesis \ref{hyp3}\label{hyp4}]\
For every closed point $x$ of $X$, there exists a $(t^\#)$ $\mathcal{O}(X)$-linear section of $t^\#_x:\mathcal{O}(X_x)\to\mathcal{O}(\mathcal{G}_{(x,X)})$  (denoted by $\beta_{x}$) such that $s^\#_x\circ\beta_{x}\circ s^\#$ and $t^\#_x\circ\beta_{x}\circ s^\#$ coincide in $\mathcal{O}(\mathcal{G}_{(x,x)})$ (equivalently, in $\mathcal{O}(\mathcal{T}_{(x,x)})$).\end{hyp}

\begin{remark} This is satisfied if, for each $x$, the groupoid $\mathcal{G}_{(x,x)}$ over $X_x$ satisfies Hypothesis \ref{hyp3}. Hypothesis \ref{hyp4} seems to be weaker in general.\end{remark}

We run through the previous argument, starting from the composition $\alpha_{x}$ given by \begin{align}V\xrightarrow{m}V\otimes\mathcal{O}(\mathcal{T})\xrightarrow{id\otimes f} V\otimes\mathcal{O}(\mathcal{G})\xrightarrow{id\otimes\beta_{x}}V_x.\end{align}We have that $\alpha_{x}$ coincides with the localization map when restricted to $V^\mathcal{G}$. Also arguing as before we have that \begin{align}m_x\circ\alpha_x=(id\otimes id\otimes\beta_x)\circ(id\otimes id\otimes f)\circ(id\otimes\Delta)\circ m,\end{align}where $m_x$ denotes the comultiplication map $V_x\to V\otimes\mathcal{O}(\mathcal{T}_{(X,x)})$. We conclude as before that $m_x\circ\alpha_{x}(V)\subset V_x\otimes 1$ after passing to $V\otimes\mathcal{O}(\mathcal{T}_{(x,x)})$. Since $\mathcal{T}_{(x,x)}$ is a groupoid over $X_x$, and $V_x$ its representation, it must be that $m_x\circ\alpha_x(v)=v\otimes 1$ in $V\otimes\mathcal{O}(\mathcal{T}_{(x,x)})=V_x\otimes_{\mathcal{O}(X_x)}\mathcal{O}(\mathcal{T}_{(x,x)})$ for any $v\in V$. So for $v\in V^\mathcal{G}$, we have $m(v)=m_x\circ\alpha_x(v)=v\otimes 1$ in $V\otimes\mathcal{O}(\mathcal{T}_{(x,x)})$. Since this is true for all $x$, we get finally that $m(v)=v\otimes 1$ in $V\otimes\mathcal{O}(\mathcal{T}_U)$ for some open neighborhood $U$ of the diagonal $X\subset X^2$, whenever $v\in V^\mathcal{G}$.

To extend this equality over $X^2$, we will need some further hypothesis. There are probably several options, but here is a natural choice:

\begin{hyp}\label{hyp5}There exist closed subschemes $\mathcal{R}_i$ of $\mathcal{T}$ such that:\begin{enumerate}\item The projection map $\mathcal{G}\times_\mathcal{T}\mathcal{R}_i\to\mathcal{R}_i$ induces a universally injective map of $\cO(X)$-modules for each $i$, and\
\item $\mathcal{T}_U$ and the various $\mathcal{R}_i$ \emph{generate} $\mathcal{T}$.\end{enumerate}\end{hyp}
Here the second condition means precisely that the multiplication maps $A_1\times_X\ldots\times_X A_n\to\mathcal{T}$, where $n$ ranges from $1$ to $\infty$ and the $A_j$ range over $\mathcal{T}_U$ and the various $\mathcal{R}_i$ (allowing repeats), induce jointly universally injective maps of $\cO(X)$-modules.

The first condition guarantees that the map $V\otimes\mathcal{O}(\mathcal{R})\to V\otimes\mathcal{O}(\mathcal{G}\times_\mathcal{T}\mathcal{R})$ is injective; it follows that for any $v\in V^\mathcal{G}$, $m(v)$ and $v\otimes 1$ have the same image in $V\otimes\mathcal{O}(\mathcal{R})$. The second condition gives that the various compositions 
\begin{align}V\otimes\mathcal{O}(\mathcal{T})\xrightarrow{id\otimes\Delta^{n-1}}V\otimes\mathcal{O}(\mathcal{T})\otimes\ldots\otimes\mathcal{O}(\mathcal{T})\to V\otimes\mathcal{O}(A_1)\otimes\ldots\otimes\mathcal{O}(A_n)\end{align}are jointly injective. If $v\in V^\mathcal{G}$ then the image of $m(v)$ under any one of these compositions certainly coincides with $v\otimes 1\otimes\ldots\otimes 1$, and hence $m(v)=v\otimes 1$ as required.

\section{Ind-schemes}

\subsection{}In this section, we develop the theory of groupoid ind-schemes to the point where we are able to formulate appropriate replacements for the hypotheses, arguments and conclusions of the previous section.

\subsection{}Consider the collection of non-empty countable cofiltered systems of $\mathcal{O}(X)$-modules with surjective transition maps. These form an additive category, denoted $PQCoh(X)$, where by definition \begin{align}Hom_{PQCoh(X)}((V_j)_{j\in J},(W_k)_{k\in K})=\nlim_{k}\colim_{j}Hom_{QCoh(X)}(V_j,W_k).\end{align}
An equivalent, and useful, way to think about this is as follows. A morphism $(V_j)_{j\in J}\to (W_k)_{k\in K}$ consists of the following data:\begin{enumerate}\item For each $k\in K$, a cofinal subsystem $S(k)$ of $J$, satisfying $S(k)\subset S(k')$ whenever $k\to k'$;\
\item For each $j\in S(k)$, a morphism $f_j^k:V_j\to W_k$, such that the diagram\begin{align}\begin{matrix}V_j & \to & W_k\\
\downarrow & & \downarrow\\
V_{j'} & \to & W_{k'}\end{matrix}\end{align}commutes whenever it exists.\end{enumerate}
We take such data up to equivalence; two such data $(S,f), (S',{f'})$ are equivalent if for every $k$, every $j\in S(k), j'\in S'(k)$, and every lower (upper?) bound $j''$ of $j,j'$, the diagram \begin{align}\begin{matrix} V_{j''} & \xrightarrow{} & V_j\\
\downarrow & & \downarrow{_{f_j^k}}\\
V_{j'} & \xrightarrow{(f')_{j'}^k} & W_k\end{matrix}\end{align}commutes. It is enough to check for $k$ in some cofinal subsystem of $K$, for $(j,j')$ in some cofinal subsystem of $S(k)\times S'(k)$, and for $j''$ being any one (rather than all) lower bound(s) of $j,j'$.

Yet another way to think of this is as follows: we view $\nlim_{j}V_j$, $\nlim_kV_k$ as topological $\mathcal{O}(X)$-modules (with the pro-discrete topology) and then $Hom_{PQCoh(X)}((V_j),(W_k))$ is none other than the set of continuous morphisms between these topologized limits. For this, the countability is essential: a countable cofiltered system $(V_j)$ admits a cofinal inverse (i.e. ordered as $\mathbb{N}$) subsystem, and consequently each map $\nlim_jV_j\to V_j$ is surjective if the transition morphisms are. To present a projectively discrete topological $\mathcal{O}(X)$-module as the limit of an object of $PQCoh(X)$ is to give a countable cofinal subsystem of its lattice of open submodules. However we will not usually think of $PQCoh(X)$ in this way, preferring to reserve the notation $\nlim$ for the functor\begin{align}\nlim:PQCoh(X)\to QCoh(X).\end{align}\noindent We note that $\nlim$ is right adjoint to the functor $QCoh(X)\to PQCoh(X)$ which takes a quasi-coherent sheaf to the corresponding single-object cofiltered system. $\nlim$ is faithful.\\

\noindent Perhaps the most useful way to think of this is given by the following:\begin{lem}\begin{enumerate}\item Every object of $PQCoh(X)$ is isomorphic to an inverse (i.e. ordered as $\mathbb{N}$) system;\
\item Let $V$ be an object of $PQCoh(X)$, let $(W_k)_{k\in\mathbb{N}}$ be an inverse system in $PQCoh(X)$, and let $V\to W$ be a morphism; then there exists an isomorphism $(U_i)_{i\in\mathbb{N}}\to V$ such that the composition $(U_i)\to V\to (W_k)$ is equivalent to a map of inverse systems in the traditional sense. In other words, writing $(f,S)$ for said composition, we may take $S(k)=[k,\infty)$ for all $k$. In other other words, the composition $(U_i)\to V\to (W_k)$ is equivalent to a surjective inverse system of morphisms.\end{enumerate}\end{lem}\noindent It follows of course that any sequence in $PQCoh(X)$ is isomorphic to an inverse system of sequences.

\subsection{}$PQCoh(X)$ is not abelian, but it is exact. First we describe the admissible sequences:

\begin{lem}Let $0\to U\to V\to W \to 0$ be a sequence in $PQCoh(X)$ which is isomorphic to an inverse system of of short exact sequences. Then $0\to\nlim U\to\nlim V\to \nlim W\to 0$ is exact.\end{lem}\noindent We call such sequences \emph{Mittag-Leffler}. To justify the name, observe first that given a morphism $V\to W$ in $PQCoh(X)$, presented as $(V_i)\xrightarrow{(S,f)}(W_j)$ say, the property that every extant $f_i^j$ is surjective is independent of the presentation. These are precisely the epimorphisms (epis) in $PQCoh(X)$. Next, we observe that an epimorphism $V\to W$ may be extended to a Mittag-Leffler sequence $U\to V\to W$ if and only if for every (equivalently, some) presentation of $V\to W$ as a surjective inverse system of morphisms $V_i\to W_i$, the resulting inverse system $ker(V_i\to W_i)$ satisfies the Mittag-Leffler condition. (In that case, to construct $U$ we take the stabilization of the pointwise kernel of any presentation of $V\to W$ as a surjective inverse system of morphisms). We note also that monomorphisms are the same as morphisms which give injections in the limit.

\subsection{}The class of Mittag-Leffler sequences is an exact structure for $PQCoh(X)$ (i.e. $PQcoh(X)$ has a unique structure of exact category in which the admissible sequences are precisely Mittag-Leffler sequences). This is a simple excersise in diagram chasing. $PQCoh(X)$ has all cokernels, and every cokernel map is admissible. $PQCoh(X)$ also has all kernels, and every kernel map is admissible. If $f:V\to W$ has kernel $K$ and cokernel $C$ then the map 
$$
coim(f):=coker(K\to V)\to ker(W\to C)=:im(f)
$$
is always both epi and mono, but it is not always an isomorphism. It is an isomorphism if for some (equivalently any) presentation of $f:V\to W$ as a surjective inverse system of morphisms $V_i\to W_i$, the inverse system $ker(V_i\to W_i)$ satisfies the Mittag-Leffler condition (but not conversely!); in that case $K$ is the stabilization of $ker(V_i\to W_i)$. We call such $f$ admissible.

\subsection{}If $X\to Y$ then we have the pushforward functor $PQCoh(X)\to PQCoh(Y)$. This functor preserves kernels and cokernels, kills no objects, and reflects isomorphisms. Consequently it preserves and reflects monos, epis, and admissible morphisms. In particular it preserves and reflects Mittag-Leffler sequences.

\subsection{}We say an object of $PQCoh(X)$ is flat, coherent etc. if it is isomorphic to a (surjective countable) cofiltered system of flat, coherent etc. $\mathcal{O}(X)$-modules. Sometimes (e.g. `coherent') this property is independent of the presentation, but more usually (e.g. `flat') it depends very much on the presentation.

\subsection{}$PQCoh(X)$ is monoidal: we set $(V_j)_{j\in J}\otimes(W_k)_{k\in K}=(V_j\otimes W_k)_{(j,k)\in J\times K}$. The one-object cofiltered system $\mathcal{O}(X)$ is the unit. It is convenient to note that for inverse systems $(V_j)_{j\in\mathbb{N}}, (W_k)_{k\in\mathbb{N}}$, we have $(V_j)\otimes(W_k)\cong (V_j\otimes W_j)_{j\in\mathbb{N}}$.

For instace, let $A$ be a flat coherent (i.e. finite rank projective) sheaf on $X$, regarded as a single-object cofiltered system in $PQCoh(X)$. Since $A$ is dualizable, $A\otimes(-)$ is a right adjoint and thus we have $A\otimes\nlim(V)=\nlim(A\otimes V)$ for any object $V$ of $PQCoh(X)$. Consequently, if more generally $A$ is a flat coherent object of $PQCoh(X)$, and $V$ is any object of $PQCoh(X)$, we have $\nlim(A\otimes\nlim(V))=\nlim(A\otimes V)$. The same formulas hold if $V$ is replaced with a countable cofiltered system with not necessarily surjective morphisms, from which point $(3)$ of the following otherwise easy lemma follows:

\begin{lem}Let $A$ be an object of $PQCoh(X)$. Then \begin{enumerate}\item $A\otimes(-)$ preserves cokernels, hence epis;\
\item If $A$ is flat then $A\otimes(-)$ preserves also admissible morphisms and their kernels, hence Mittag-Leffler sequences and admissible monos (but not monos or kernels in general);
\item If $A$ is flat coherent then $A\otimes(-)$ preserves also kernels, hence monos.\end{enumerate}\end{lem}

\subsection{}An affine ind-scheme over $X$ is by definition a (surjective, countable) cofiltered system of $\mathcal{O}(X)$-algebras\footnote{This definition is more restrictive than being simply a ring object in $PQCoh(X)$.}. Affine ind-schemes form a category, $IAff_X$, where by definition a morphism in $IAff_X^{op}$, between $(A_i)$ and $(B_j)$ say, is any morphism $(S,f)$ in $PQCoh(X)$ for which every extant $f_i^j$ is a ring map\footnote{Equivalently, noting that the limit of an affine ind-scheme is naturally a ring, we see that morphisms between affine ind-schemes are exactly those morphisms in $PQCoh(X)$ whose limit is a ring homomorphism. This shows, for instance, that the forgetful functor $IAff_X^{op}\to PQCoh(X)$ reflects isomorphisms. To present a projectively discrete topological ring as the limit of an object of $IAff_X$ is to give a countable cofinal subsystem of ideals in its lattice of open submodules.}. An affine ind-scheme is called flat, coherent etc. if it is so as an object of $PQCoh(X)$ \footnote{I do not know (nor care) whether a flat affine ind-scheme may be presented as a surjective cofiltered system of flat $\cO(X)$-algebras.}. $IAff_X$ inherits the monoidal structure from $PQCoh(X)$. A groupoid ind-scheme is a groupoid in $IAff$ \footnote{$IAff$ is meant as a stack over $Aff$; this is just a convenient way of saying that a groupoid ind-scheme is an object of $IAff_{X\times X}$ for some $X$ with the appropriate operations between its two projections to $IAff_X$.}. As usual, for such a thing $\mathcal{G}$, we will write $\mathcal{O}(\mathcal{G})$ for the corresponding object of $IAff_X^{op}$, and keep the notations $\Delta, \eta, t^\#, s^\#$ of the previous section. Since $\eta, s^\#$ split each other, they are respectively admissible epi, mono. Similarly $t^\#$ and $\Delta$ are admissible mono. A (right) representation of the groupoid ind-scheme $\mathcal{G}$ is a quasi-coherent sheaf $V$ on $X$ together with a morphism \begin{align}m:V\to V\otimes\mathcal{O}(\mathcal{G})\end{align}satisfying the natural comodule axioms. Representations of $\mathcal{G}$ form a category, denoted $Rep_X(\mathcal{G})$, in the obvious manner. Of course, one could make the same definition with $V$ being an arbitrary object of $PQCoh(X)$; however, this would apparently make it difficult for $Rep _X(\mathcal{G})$ to be abelian.

\subsection{Coinduction.}We assume that $t$ is finite flat. In that case, the forgetful functor $Rep_X(\mathcal{G})\to QCoh(X)$ has the all-important right adjoint\begin{align}\begin{matrix}Coind_X^\mathcal{G}:&QCoh(X)&\to& Rep_X(\mathcal{G})\\
&V&\mapsto& \nlim(V\otimes\mathcal{O}(\mathcal{G}))\end{matrix}\end{align}where the structure of representation on $\nlim(V\otimes\mathcal{O}(\mathcal{G}))$ is given as follows. Certainly there is a map $V\otimes\mathcal{O}(\mathcal{G})\xrightarrow{id\otimes\Delta}V\otimes\mathcal{O}(\mathcal{G})\otimes\mathcal{O}(\mathcal{G})$; taking $\nlim$ we get $\nlim(V\otimes\mathcal{O}(\mathcal{G}))\xrightarrow{}\nlim(V\otimes\mathcal{O}(\mathcal{G})\otimes\mathcal{O}(\mathcal{G}))=\nlim(\nlim(V\otimes\mathcal{O}(\mathcal{G}))\otimes\mathcal{O}(\mathcal{G}))$ since $\mathcal{O}(\mathcal{G})$ is flat coherent with respect to $t^\#$. By the adjunction property of $\nlim$, this is the same as a map $\nlim(V\otimes\mathcal{O}(\mathcal{G}))\to \nlim(V\otimes\mathcal{O}(\mathcal{G}))\otimes\mathcal{O}(\mathcal{G})$. Having constructed the map, it is easy to see that it satisfies the comodule axioms, and that $Coind_X^\mathcal{G}$ is indeed right adjoint to the forgetful functor.

\subsection{Abelian-ness.}It follows from the flatness of $t$ that $Rep_X^\mathcal{G}$ is abelian. The proof is more or less the same as in the scheme case, but we'll indicate it anyway. Suppose $U\to V\to W\to 0$ is an exact sequence in $QCoh(X)$. If $U\to V$ lifts to $Rep_X(\mathcal{G})$, then $V\to V\otimes\mathcal{O}(\mathcal{G})\to W\otimes\mathcal{O}(\mathcal{G})$ factors through $W$, and one checks that this defines a comodule structure on $W$ (independently of the assumption on $t$), and that $V\to W$ is a map of comodules, and that this is the unique way to lift $V\to W$ to $Rep_X^\mathcal{G}$. If instead $U\to V$ is injective and $V\to W$ lifts to $Rep_X(\mathcal{G})$, then consider the diagram \begin{align}\begin{matrix}0&&0&&0\\
\downarrow&&\downarrow&&\downarrow\\
U &\dashrightarrow & U\otimes\mathcal{O}(\mathcal{G}) & \xrightarrow{\dashrightarrow}& U\otimes\mathcal{O}(\mathcal{G})\otimes\mathcal{O}(\mathcal{G})\\
\downarrow&&\downarrow&&\downarrow\\
V &\xrightarrow{~~~~} & V\otimes\mathcal{O}(\mathcal{G}) & \xrightarrow{\xrightarrow{~~~}}& V\otimes\mathcal{O}(\mathcal{G})\otimes\mathcal{O}(\mathcal{G})\\
\downarrow&&\downarrow&&\downarrow\\
W &\xrightarrow{~~~~} & W\otimes\mathcal{O}(\mathcal{G}) & \xrightarrow{\xrightarrow{~~~}}& W\otimes\mathcal{O}(\mathcal{G})\otimes\mathcal{O}(\mathcal{G})\\
\downarrow&&\downarrow&&\downarrow\\
0&&0&&0\end{matrix}\end{align}By the flatness of $t$, each column is Mittag-Leffler. Therefore in the limit, the columns become short exact sequences. The leftmost dashed arrow exists in the limit, and hence exists outright by adjunction ($U$ being already quasi-coherent). Hence both dashed arrows exist outright. All necessary commutativity/equalizing properties follow from the corresponding limiting statements, since $\nlim$ is faithful. A similar diagram yields the counity condition, and thus one obtains the (unique) lifting of $U\to V$ to $Rep_X(\mathcal{G})$. From this it is a formal consequence that $Rep_X^\mathcal{G}$ is abelian, and that the forgetful functor to $QCoh(X)$ is exact and faithful.

\subsection{Descent.}By Barr-Beck, one obtains that $Rep_X(\mathcal{G})$ is equivalent to the category of comodules in $QCoh(X)$ for the comonad $\nlim((-)\otimes\mathcal{O}(\mathcal{G}))$. In order for Lemma \ref{lemma1} to go through, one must make the additional assumption that $f:Y\to X$ is finite (as well as faithfully flat), so that the comonads $f^*\circ Res^\mathcal{G}_X\circ Coind^\mathcal{G}_X\circ f_*$ and $Res_Y^{Y\times\mathcal{G}\times Y}\circ Coind_Y^{Y\times\mathcal{G}\times Y}$ on $QCoh(Y)$ coincide.

\subsection{}The remaining arguments of the previous section apply more or less verbatim. We will point out (in chronological order) the points where some extra thought is needed:\begin{enumerate}\item $C$ is defined the same way as before (it is a scheme).\
\item $\mathcal{O}(\mathcal{T})$ is defined as the subsystem of $\mathcal{O}(\mathcal{G})$ generated by the images of $s^\#$, $t^\#$, which is automatically surjective; it is left as an exercise that $\mathcal{T}$ inherits the groupoid structure from $\mathcal{G}$. From the point of view of exact categories, $\mathcal{O}(\mathcal{T})$ is the kernel of the cokernel morphism $\mathcal{O}(\mathcal{G})\to coker(\mathcal{O}(X\times X)\to\mathcal{O}(\mathcal{G}))$; since cokernel maps are admissible, we get that $\mathcal{O}(\mathcal{T})\to\mathcal{O}(\mathcal{G})$ is admissible.\
\item $\mathcal{I}$ is defined as $\mathcal{I}=\mathcal{G}\times_{X\times X}X$; on level of surjective systems, it is the system obtained by quotienting out each piece of $\mathcal{O}(\mathcal{G})$ by the ideal generated by the image of $(s^\#-t^\#)\mathcal{O}(X)$. Those ideals also form a surjective system, denoted $M$, and so we have the Mittag-Leffler sequence $M\to\mathcal{O}(\mathcal{G})\to\mathcal{O}(\mathcal{I})$.\
\item Hypothesis \ref{hyp} needs changing slightly. First we replace `flat' by `finite flat'. Next note that the equalizer of $\mathcal{O}(\mathcal{G})\rightrightarrows\mathcal{O}(\mathcal{G})\otimes\mathcal{O}(\mathcal{I})$ is equal to the kernel of the composition $\mathcal{O}(\mathcal{G})\xrightarrow{\Delta-id\otimes1}\mathcal{O}(\mathcal{G})\otimes\mathcal{O}(\mathcal{G})\to\mathcal{O}(\mathcal{G})\otimes\mathcal{O}(\mathcal{I})$ and is an ind-scheme (the forgetful functor $IAff_X^{op}\to PQCoh(X)$ has and preserves equalizers of equalizers it creates). Moreover since $M\to\mathcal{O}(\mathcal{G})\to\mathcal{O}(\mathcal{I})$ is Mittag-Leffler, so is $\mathcal{O}(\mathcal{G})\otimes M\to\mathcal{O}(\mathcal{G})\otimes\mathcal{O}(\mathcal{G})\to\mathcal{O}(\mathcal{G})\otimes\mathcal{O}(\mathcal{I})$, and it follows that the equalizer in question is the Cartesian product of \begin{align}\label{carty}\begin{matrix}&&\mathcal{O}(\mathcal{G})\\
&&~~~~~~~~~~~\downarrow{_{\Delta-id\otimes1}}\\
\mathcal{O}(\mathcal{G})\otimes M&\to&\mathcal{O}(\mathcal{G})\otimes\mathcal{O}(\mathcal{G}).\end{matrix}\end{align}\noindent Therefore the condition that $\mathcal{G}\times_X\mathcal{I}\rightrightarrows\mathcal{G}\to\mathcal{T}$ is a coequalizer in $IAff_X$ is equivalent to the condition that $\mathcal{O}(\mathcal{T})$ maps isomorphically to the Cartesian product of that diagram in $PQCoh(X)$, or equivalently that it maps isomorphically to the kernel of $\mathcal{O}(\mathcal{G})\to\mathcal{O}(\mathcal{G})\otimes\mathcal{O}(\mathcal{I})$ in $PQCoh(X)$. We require the additional hypothesis that this map is \textbf{admissible} (recall that $\mathcal{O}(\mathcal{T})\to\mathcal{O}(\mathcal{G})$ is admissible). This hypothesis is equivalent to that in any presentation of diagram \ref{carty} as a surjective inverse systems of diagrams the pointwise pullback satisfies the Mittag-Leffler condition and $\mathcal{O}(\mathcal{T})$ maps isomorphically to its stabilization.\
\item Lemma \ref{lemma2} goes through as written if one understands `injective' as mono, and recalling that $\mathcal{O}(\mathcal{T})\to\mathcal{O}(\mathcal{G})$ is the kernel of the admissible morphism $\mathcal{O}(\mathcal{G})\xrightarrow{\Delta-id\otimes1}\mathcal{O}(\mathcal{G})\otimes\mathcal{O}(\mathcal{I})$ (and these properties are preserved when tensoring with flat objects).\
\item Hypothesis \ref{hyp2} has already been made. The construction of comodule structure on $Hom_{\mathcal{O}(X)}(V,V')$ is the same, noting that $Hom_{PQCoh(X)}(V,V'\otimes \mathcal{O}(\mathcal{T}))=\nlim Hom_{QCoh(X)}(V,V'\otimes\mathcal{O}(\mathcal{T})_i)$ (choosing a presentation of $\mathcal{O}(\mathcal{T})$), which equals $\nlim (Hom_{QCoh(X)}(V,V')\otimes\mathcal{O}(\mathcal{T})_i)$, and thus $Hom_{PQCoh(X)}(V,V'\otimes \mathcal{O}(\mathcal{T}))$ is `internalized' as $Hom_{Qcoh(X)}(V,V')\otimes \mathcal{O}(\mathcal{T})$. The waffle about high roads and low roads works out the same and one obtains Equation \ref{eq9}, whose RHS is interpreted in exact categories-speak as the kernel of $Hom_{QCoh(X)}(V,V')\xrightarrow{m-id\otimes1}Hom_{QCoh(X)}(V,V')\otimes\mathcal{O}(\mathcal{T})$. Since $m,id\otimes 1$ have the common section $id\otimes\eta$, we also get Equation \ref{eq10} (interpreted as a pullback in the obvious way). Thus again we will deduce fullness from the `reflects invariants' property.\
\item We make Hypothesis \ref{hyp3} and define $\alpha, \beta$ as before. It is still true that $\alpha$ is a section of the inclusion $V^{\mathcal{T}}\to V$, but one must be a little more careful. Indeed since the map $\mathcal{O}(X\times X)\to\mathcal{O}(\mathcal{T})$ is epi, the map $(id\otimes \beta)\circ(id\otimes f)\circ\Delta:\mathcal{O}(\mathcal{T})\to\mathcal{O}(\mathcal{T})$ factors through the image of $t^\#\otimes(s^\#\circ\beta\circ s^\#):\mathcal{O}(X\times X)\to\mathcal{O}(\mathcal{T})$. This latter map equals $t^\#\otimes(t^\#\circ\beta\circ s^\#)$ which of course factors through $t^\#:\mathcal{O}(X)\to\mathcal{O}(\mathcal{T})$. This last map is an isomorphism with its image, so we see that $(id\otimes \beta)\circ(id\otimes f)\circ\Delta$ factors through $t^\#$. Likewise the map $(id\otimes id\otimes\beta)\circ(id\otimes id\otimes f)\circ(id\otimes\Delta):V\otimes\mathcal{O}(\mathcal{T})\to V\otimes\mathcal{O}(\mathcal{T})$ factors through $id\otimes 1:V\to V\otimes\mathcal{O}(\mathcal{T})$ as required.\
\item We make the same alternative hypothesis as Hypothesis \ref{hyp4}, and make the same conclusion that if $v\in V^{\mathcal{G}}$ then $m(v)=v\otimes1$ in $V\otimes\mathcal{O}(\mathcal{T}_{(x,x)})$ for all closed points $x\in X$. Of course it does not necessarily follow that we have the equality in $V\otimes\mathcal{O}(\mathcal{T}_U)$ for some Zariski-open neighborhood $U$ of the diagonal $X\subset X\times X$, but it does hold for $U$ being the `complement' of some closed ind-subscheme of $X\times X$, which suffices.\
\item From Hypothesis \ref{hyp5} onwards the argument is identical, reading `sub-ind-scheme' for subscheme and `mono` for injective.\end{enumerate}

In summary, we have the following:

\begin{prop}\label{bigprop}Let $\cG$ be an affine groupoid ind-scheme over the affine base scheme $X$, with adjacency groupoid $\cT$ and isotropy subgroup $\cI$. 
\begin{enumerate}\item Suppose the following conditions hold:
\begin{enumerate}\item Both $\cG$, $\cT$ are finite flat over $X$ with respect to both the head and tails maps;
\item $\cO(\cG)\to\cO(\cG)\otimes\cO(\cI)$ is admissible and $\cO(\cT)$ is its kernel;
\item For every closed point $x$ of $X$, there exists a $(t^\#)$ $\mathcal{O}(X)$-linear section of $t^\#_x:\mathcal{O}(X_x)\to\mathcal{O}(\mathcal{G}_{(x,X)})$  (denoted by $\beta_{x}$) such that $s^\#_x\circ\beta_{x}\circ s^\#$ and $t^\#_x\circ\beta_{x}\circ s^\#$ coincide in $\mathcal{O}(\mathcal{G}_{(x,x)})$ (equivalently, in $\mathcal{O}(\mathcal{T}_{(x,x)})$).\end{enumerate}
Then the functor of restriction $Rep_X(\cT)\to Rep_X(\cG)$ reflects invariants in some neighborhood $U$ of the diagonal $X\subset X^2$.\
\item Suppose in addition that there exist closed sub-ind-schemes $\mathcal{R}_i$ of $\mathcal{T}$ such that:\begin{enumerate}\item The map $\mathcal{G}\times_\mathcal{T}\mathcal{R}_i\to\mathcal{R}_i$ induces a universally injective of $\cO(X)$-modules for each $i$, and\
\item $\mathcal{T}_U$ and the various $\mathcal{R}_i$ generate $\mathcal{T}$.\end{enumerate}
Then $Rep_X(\cT)\to Rep_X(\cG)$ is full.\
\item The functor $Rep_X(\cT)\to Rep_X(\cG)^{\cI}$ is an equivalence on the full subcategories of objects which are flat over $X$. Noting that $\cO(X)$-flat (even projective) resolutions exist in $Rep_X(\cT)$, one obtains the equivalence
$$
Rep_X(\cT)\cong Rep_X(\cG)^{``\cI"}
$$where the latter category denotes the full subcategory of $Rep_X(\cG)$ consisting of those objects which admit resolutions by $\cO(X)$-flat objects which have trivial isotropy.\end{enumerate}\end{prop}

\begin{remark}We think of the above isotropy condition as `having trivial derived isotropy'. We expect that the better way to phrase it is to replace $\cI$ by the natural groupoid ind-dga-scheme (whatever that means!), denoted $``\cI"$, at which point the condition may be literally interpreted as having trivial $``\cI"$-action. We do not pursue this here.\end{remark}

\section{Proof of Theorems \ref{thm1} and \ref{thm2?}}\label{checktion}

For background material on root systems and reflection groups, as used heavily in paragraph \ref{fpd}, see \cite{humph}.

\subsection{Adjoint case.}\label{adj}Let us begins with the case where $G$ is adjoint. Then $\pi_1(G^\vee)$ is trivial, $\widetilde{W}^{aff}=W^{aff}$ and $\cT=\spec(H^\bullet_{G^\vee(\cO)\rtimes \bG_m}(Fl))$ is really the adjacency groupoid of $\cG=\widetilde{W}^{aff}\#\liet^*$. We simply check the conditions of Proposition \ref{bigprop}.

\begin{enumerate}\item Certainly $s$, $t$ are finite flat (for $\mathcal{T}$, $\mathcal{G}$).\
\item $\mathcal{O}(\mathcal{T})\to\mathcal{O}(\mathcal{G})$ is the kernel of the admissible morphism $\mathcal{O}(\mathcal{G})\xrightarrow{\Delta-id\otimes 1}\mathcal{O}(\mathcal{G})\otimes\mathcal{O}(\mathcal{I})$. Indeed by \cite{GKM}, Theorem 1.2.2 we have that for an equivariantly formal $T$-variety $X$ with finitely many fixed points,
\begin{align}\label{locali} H^\bullet_T(X)=ker(H^\bullet_T(X^T)\to \prod_{j}H^\bullet_T(X_j)),\end{align}
where the $X_j$ are the one-dimensional orbits of $T$ on $X$ and the $j$-component of the map is $(\xi_x)_{x\in X^T}\mapsto \xi_{j_0}-\xi_{j_\infty}$ (on ${\lie(\stab_T(X_j))}$). In our situation, the one-dimensional $T^\vee\times\mathbb{G}_m$-orbits on $\overline{Fl_\lambda}$ correspond (up to equivalence, i.e. repetition in the above morphism) to pairs $(s,\gamma)$, where $s$ is an affine reflection and $\gamma$ is any fixed point in $\overline{Fl_\lambda}$ such that $\gamma s$ is also in $\overline{Fl_\lambda}$. Since $T^\vee$ has the same fixed point set in $\overline{Fl_\lambda}$ as $T^\vee\times\bG_m$, each $H^\bullet_{T^\vee\times\bG_m}(\overline{Fl_\lambda})$ is free over $\bA^1$. Since $\bA^1$ has homological dimension $1$, one may set $\hbar=1$ in Equation \ref{locali} (and obtain a correct formula). Then $H^\bullet_{T^\vee\times\mathbb{G}_m}(\overline{Fl_\lambda}_{(s,\gamma)})_{\hbar=1}=\mathcal{O}(\Gamma_\gamma|_{{(\liet^*)}^s})$. Writing $\mathcal{O}(\mathcal{G})_\lambda:=  H^\bullet_{T^\vee\times\mathbb{G}_m}(\overline{Fl_\lambda}^{T^\vee})_{\hbar=1}$, $\mathcal{O}(\mathcal{T})_\lambda:=  H^\bullet_{T^\vee\times\mathbb{G}_m}(\overline{Fl_\lambda})_{\hbar=1}$, we have\begin{align}\begin{matrix} \mathcal{O}(\mathcal{G})_\lambda& = & \prod_{\gamma \in \overline{Fl_\lambda}^{T^\vee}}\mathcal{O}(\Gamma_\gamma)\\
\mathcal{O}(\mathcal{I})_\lambda& = & \prod_{\gamma \in \overline{Fl_\lambda}^{T^\vee}}\mathcal{O}(\Gamma_\gamma|_{{(\liet^*)}^\gamma})\\
\mathcal{O}(\mathcal{G})_\lambda\otimes \mathcal{O}(\mathcal{I})_\mu &=& \prod_{(\gamma_0,\gamma_\infty)\in \overline{Fl_\lambda}^{T^\vee}\times \overline{Fl_\mu}^{T^\vee}}\mathcal{O}(\Gamma_{\gamma_0}|_{{(\liet^*)}^{\gamma_\infty}}).\end{matrix}\end{align}In the limit, we have $\nlim\mathcal{O}(\mathcal{G})=\prod_{\gamma\in Fl^{T^\vee}}\mathcal{O}(\Gamma_\gamma)$ and $\nlim(\mathcal{O}(\mathcal{G})\otimes\mathcal{O}(\mathcal{I}))=\prod_{\gamma_0,\gamma_\infty\in Fl^{T^\vee}}\mathcal{O}(\Gamma_{\gamma_0}|_{{(\liet^*)}^{\gamma_\infty}})$. Both are equipped with the product topology, and the morphism $\Delta-id\otimes 1$ between them sends $(\xi_\gamma)_\gamma$ to $(\xi_{\gamma_0.\gamma_\infty}|_{(\liet^*)^{\gamma_\infty}}-\xi_{\gamma_0}|_{(\liet^*)^{\gamma_\infty}})_{(\gamma_0,\gamma_\infty)}$. For each $\lambda$, the set $S_\lambda$ of pairs $\gamma_0,\gamma_\infty$ such that both $\gamma_0$ and $\gamma_0.\gamma_\infty$ are contained in $\overline{Fl^\lambda}$ is finite, and so one obtains the discrete quotient $\prod_{(\gamma_0,\gamma_\infty)\in S_\lambda}\mathcal{O}(\Gamma_{\gamma_0}|_{{(\liet^*)}^{\gamma_\infty}})$ of $\nlim(\mathcal{O}(\mathcal{G})\otimes\mathcal{O}(\mathcal{I}))$. These quotients are cofinal in the cofiltered system of all discrete quotients, and thus may be used as a presentation. But then composition of $\Delta-id\otimes1$ with projection to the $\lambda$-piece of this presentation factors through $\mathcal{O}(\mathcal{G})_\lambda$, and yields the the map\begin{align}\label{huh}\begin{matrix}\prod_{\gamma \in \overline{Fl_\lambda}^{T^\vee}}\mathcal{O}(\Gamma_\gamma) & \to & \prod_{(\gamma_0,\gamma_\infty)\in S_\lambda}\mathcal{O}(\Gamma_{\gamma_0}|_{{(\liet^*)}^{\gamma_\infty}})\\
(\xi_\gamma)_\gamma&\mapsto &(\xi_{\gamma_0.\gamma_\infty}|_{(\liet^*)^{\gamma_\infty}}-\xi_{\gamma_0}|_{(\liet^*)^{\gamma_\infty}})_{(\gamma_0,\gamma_\infty)}\end{matrix}\end{align}If one projects the right-hand side to the product of all those factors for which $\gamma_\infty$ is an affine reflection, then one obtains the map whose kernel is $\mathcal{O}(\mathcal{T})_\lambda$, according to Equation \ref{locali}. Therefore $\mathcal{O}(\mathcal{T})_\lambda$ is also the kernel of Equation \ref{huh}, which is to say that $\mathcal{O}(\mathcal{T})$ is the kernel of the admissible morphism $\mathcal{O}(\mathcal{G})\xrightarrow{\Delta-id\otimes 1}\mathcal{O}(\mathcal{G})\otimes\mathcal{O}(\mathcal{I})$, as required.\
\item Hypothesis \ref{hyp4} holds. Indeed, note that for any closed point $x\in\liet^*$, the stabilizer $\widetilde{W}^{aff}_x$ is finite, and so one may define a map \begin{align}\begin{matrix}\beta_x:&\mathcal{O}(\mathcal{G}_{(x,X)})=\prod_{\gamma\in \widetilde{W}^{aff}}\mathcal{O}({\Gamma_\gamma}_{(x,X)})& \to & \mathcal{O}(\liet^*_x)\\
&~~~~~~~~~~~~~~~~~~~~~~~~~~~~~~~~~~~~~~~~~~~~~~~~~~~~~~(\xi_\gamma)_\gamma & \mapsto & \frac{1}{|\widetilde{W}^{aff}_x|}\sum_{\gamma\in \widetilde{W}^{aff}_x}{\pi_1}_*(\xi_\gamma).\end{matrix}\end{align}$\beta_x$ is a $t^\#$-linear section of $t^\#_x$. The map $\beta_x\circ s^\#:\mathcal{O(\liet^*)}\to\mathcal{O}(\liet^*_x)$ is the composition of the averaging map with respect to $\widetilde{W}^{aff}_x$ with the localization map. Therefore $s_x^\#\circ\beta_x\circ s^\#$ and $t_x^\#\circ\beta_x\circ s^\#$ coincide on every factor of the form $\mathcal{O}(\Gamma_\gamma)$ with $\gamma\in \widetilde{W}^{aff}_x$. Noting that these are the only factors which survive in $\mathcal{O}(\mathcal{G}_{(x,x)})$, we get the desired result\footnote{With a little more care, we may take $U\subset \liet^*\times\liet^*$ to be the complement of the (ind-scheme) union of the graphs of the fixed-point-free elements of $\widetilde{W}^{aff}$. This is not necessary for what follows.}.\
\item Hypothesis \ref{hyp5} holds taking the closed sub-ind-schemes $\mathcal{R}_i$ of $\mathcal{T}$ to be the graphs $\Gamma_\gamma$ ($\gamma$ ranging over all of $\widetilde{W}^{aff}$). Of course $\Gamma_\gamma$ is most naturally a closed sub-ind-scheme of $\mathcal{G}$; it is also a closed subscheme of $\mathcal{T}$, because it is a closed subscheme of $\liet^*\times\liet^*$. Then the projection $\mathcal{G}\times_\mathcal{T}\Gamma_\gamma\to \Gamma_\gamma$ is an isomorphism, so certainly induces a universally injective map of $\cO(X)$-modules. Finally, the various multiplications $\Gamma_{\gamma}\times_X\mathcal{T}_U\to\mathcal{T}$ constitute an open cover of $\mathcal{T}$, so certainly induce jointly universally injective morphisms of $\cO(X)$-modules.\end{enumerate}

\subsection{General case.}For general reductive $G$, we write $\cT^{ad}$, $\cG^{ad}$ for the groupoids associated as in the previous paragraph to its adjoint group. $\widetilde{W}^{aff}$ acts (by `conjugation') on both $\cT^{ad}$, $\cG^{ad}$ by groupoid automorphisms covering its natural action on $\liet^*$ (for both the heads and the tails map), and the map $\cG^{ad}\to\cT^{ad}$ is $\widetilde{W}^{aff}$-equivariant.  From Lemma \ref{imaget} it follows that $\cT$, $\cG$ are obtained from $\cT^{ad}$, $\cG^{ad}$ as follows:
$$
\cT = \cT^{ad}\times^{W^{aff}}\widetilde{W}^{aff}
$$
$$
\cG = \cG^{ad}\times^{W^{aff}}\widetilde{W}^{aff}.
$$
Here the symbol $\times^{W^{aff}}$ denotes the balanced product, where $W^{aff}$ acts on $\widetilde{W}^{aff}$ by left translations and on $\cT$, $\cG$ by right translations. It is left as an exercise to check that the right-hand expressions have natural groupoid structures and that the equalities are as groupoids\footnote{This is essentially the same as the original construction of the transformation groupoid $\cG=\widetilde{W}^{aff}\times\liet^*$.}.

It follows that the right adjoint of the restriction functor $Res^{\cT}_{\cT^{ad}}$,
$$
Coind^{\cT}_{\cT^{ad}}(-):=(-)\otimes^{\cO(\cT^{ad})}\cO(\cT) = V\otimes^{\cO(W^{aff})}\cO(\widetilde{W}^{aff})\simeq\bigoplus_{\pi_1(G^\vee)}V ~~\footnote{The symbols $\otimes^{\cO(\cT^{ad})}$, $\otimes^{\cO(W^{aff})}$ denote the balanced tensor product (over $\liet^*$), which is to say the subspace of the tensor product on which the two `inner' comodule structures coincide.},
$$
satisfies the equation
$$
Res_{\cT}^{\cG}\circ Coind_{\cT^{ad}}^{\cT}            \cong Coind^{\cG}_{\cG^{ad}}\circ Res_{\cT^{ad}}^{\cG^{ad}}.
$$
Both $Res_{\cT^{ad}}^{\cT}$ and $Res_{\cG^{ad}}^{\cG}$ satisfy the conditions of comonadicity (their right adjoints are the coinduction functors above; they are exact functors between abelian categories which kill no objects). Writing $C_{\cT}$, $C_{\cG}$ for the corresponding comonads, we have
$$
C_{\cG}\circ Res_{\cT^{ad}}^{\cG^{ad}}\cong Res_{\cT^{ad}}^{\cG^{ad}}\circ C_{\cT};
$$
thus $Res_{\cT^{ad}}^{\cG^{ad}}$ defines a functor $\widetilde{Res}_{\cT^{ad}}^{\cG^{ad}}:C_{\cT}-comod\to C_{\cG}-comod$ which makes the following diagram commutative:
$$
\begin{matrix}Rep_{\liet^*}(\cT)&\xrightarrow{\sim}&C_{\cT}-comod\\
\downarrow{^{Res_{\cT}^{\cG}}}&&\downarrow{^{\widetilde{Res}_{\cT^{ad}}^{\cG^{ad}}}}\\
Rep_{\liet^*}(\cG)&\xrightarrow{\sim}&C_{\cG}-comod.\end{matrix}
$$
That $\widetilde{Res}_{\cT^{ad}}^{\cG^{ad}}:C_{\cT}-comod\to C_{\cG}-comod$ is fully faithful follows from the fact established previously that $Res_{\cT^{ad}}^{\cG^{ad}}:Rep_{\liet*}(\cT^{ad})\to Rep_{\liet*}(\cG^{ad})$ is fully faithful. 

The fully faithfulness of $Res_{\cT^{ad}}^{\cG^{ad}}$ also implies that a $C_{\cG}$-comodule which is also a representation of $\cT^{ad}$ is a $C_{\cT}$-comodule in a unique way.  Equivalently, a representation of $\cG$ is the restriction of a representation of $\cT$ if and only if its restricted structure of $\cG^{aff}$-representation descends to one of $\cT^{aff}$-representation. This shows that the objects of $Toda_1(G)$ are precisely $\widetilde{W}^{aff}$-equivariant quasicoherent sheaves on $\liet^*$ which admit $W^{aff}$-equivariant resolutions by $W^{aff}$-equivariant quasicoherent sheaves which are flat over $\liet^*$ and have trivial isotropy. On the other hand, every $Toda_1(G)$-module admits a resolution by free $Toda_1(G)$-modules, which correspond to $\widetilde{W}^{aff}$-equivariant quasicoherent sheaves which are free over $\liet^*$ and have trivial isotropy in $W^{aff}$. Thus we have:

\begin{thm}There is an equivalence of categories:
$$
Rep_{\liet^*}(\cT)\xrightarrow{\sim}QCoh^{\widetilde{W}^{aff}}(\liet^*)^{``\cI"}.
$$
By definition the second category is the full subcategory of $QCoh^{\widetilde{W}^{aff}}(\liet^*)$ consisting of objects which admit resolutions by $\widetilde{W}^{aff}$-equivariant quasicoherent sheaves which are free over $\liet^*$ and have trivial isotropy in $W^{aff}$. However, to check this condition it is sufficient to find a $W^{aff}$-equivariant resolution by $W^{aff}$-equivariant quasicoherent sheaves which are flat over $\liet^*$ and have trivial isotropy.
\end{thm}

\subsection{Finite parabolic descent.}\label{fpd}In this special case, the `derived isotropy' condition has a more familiar description. Let us denote the stabilizer in $W^{aff}$ of the closed point $x\in\liet^*$ by $\Gamma^x$. Recall that $W^{aff}$ acts simply transitively on the connected components of the complement of the affine reflection hyperplanes (`hyperplanes' for short) in the real span $\liet^*_\R$ of the character lattice. These connected components are called \emph{alcoves}. It follows immediately that $\Gamma^{Re(x)}$ acts sub-simply transitively on the set of alcoves containing $Re(x)$ in their closure, so that $\Gamma^{Re(x)}$ is in particular finite. On the other hand, for any two such alcoves $P,Q$ one may draw a path between them sufficiently close to ${Re(x)}$ that the only hyperplanes it crosses pass through ${Re(x)}$; we may deform this path slightly (staying in $\liet^*_\R$) so that it does not meet any pairwise intersections of hyperplanes, and the result is a sequence $P=P_0,\ldots,P_n=Q$ of alcoves such that $P_{j-1}$, $P_j$ share a common face for each $j=1,\ldots,n$, which is contained in the hyperplane corresponding to the affine root $\alpha^j$. Then $s_{\alpha^j}P_{j-1}=P_j$, so that $s_{\alpha^n}\ldots s_{\alpha^1}P_0=P_n$. This shows that $\Gamma^{Re(x)}$ acts simply transitively on the alcoves containing ${Re(x)}$ in their closure and is generated by reflections. Moreover, the reflection $s_{\alpha^1}s_{\alpha^2}\ldots s_{\alpha^{n-1}} s_{\alpha^n} s_{\alpha^{n-1}}\ldots s_{\alpha^2} s_{\alpha^1}$ is through some face of $P$; it follows by induction on $n$ that $\Gamma^{Re(x)}$ is generated by its reflections through faces of $P$. Since $P$ is conjugate to the fundamental alcove, which is bounded by the simple root hyperplanes, it follows that $\Gamma^{Re(x)}$ is a parabolic subgroup of $W^{aff}$. On the other hand, translating by $-x$ sends the hyperplanes passing through $x$ to certain hyperplanes passing through $0$, which are still reflection hyperplanes. We have thus proved:
\begin{lem}The composition 
$$
\Gamma^{Re(x)}\to W^{aff}\xrightarrow{\pi} W
$$
is injective and realizes $\Gamma^{Re(x)}$ as a reflection subgroup of $W$.

In fact, $\Gamma^{Re(x)}$ is realized as the Weyl subgroup corresponding to some root subsystem $\Phi_x$ of the root system $\Phi$ of $W$ \footnote{$\Phi_x$ is not necessarily integrally closed in $\Phi$, nor irreducible even if $\Phi$ is; see for instance what happens in type $G_2$.}.
\end{lem}
Now we may calculate:
$$
\begin{matrix}\Gamma^x&=&\Gamma^{Re(x)}\cap \pi^{-1}Stab_W(Im(x))\\
&\xrightarrow{\sim}&\pi(\Gamma^{Re(x)}\cap \pi^{-1}Stab_W(Im(x)))\\
&=&\pi(\Gamma^{Re(x)})\cap Stab_W(Im(x))\\
&=&Stab_{\pi(\Gamma^{Re(x)})}(Im(x))\end{matrix}
$$
Let $V^x$ denote the fixed point subspace of $\pi(\Gamma^{Re(x)})$ acting on $\liet^*_\R$. $V^x$ is complementary to the span of $\Phi_x$, so that $\Phi_x+V^x$ is a root system of full rank in $\liet^*+\R/V^x$, with Weyl group $\pi(\Gamma^{Re(x)})$. We have
$$
\begin{matrix}\Gamma^x&=&Stab_{\pi(\Gamma^{Re(x)})}(Im(x))\\
&=&Stab_{\pi(\Gamma^{Re(x)})}(Im(x)+V^x)\end{matrix}
$$
which is a parabolic subgroup of $\pi(\Gamma^{Re(x)})$ \footnote{being the stabilizer of a point in the reflection representation of a Weyl group; remove the words `affine' from the discussion at the start of this paragraph.}. We have proved:
\begin{prop}The stabilizer group $\Gamma^x$ is a finite parabolic subgroup of $W^{aff}$. In particular it is generated by affine reflections passing through $x$ \footnote{It seems likely that this is well known, but I have not been able to find a reference for it.}.
\end{prop}

Moreover, every finite parabolic subgroup of $W^{aff}$ arises in this way. Next, we have (paraphrasing):
\begin{thm}[Chevalley-Shephard-Todd, \cite{C}\cite{ST}]
Let $V$ be a complex vector space and $\Gamma$ be a finite subgroup of $GL(V)$ generated by reflections. Then $\cO(V)$ is free of finite rank over  $\cO(V//\Gamma)$.
\end{thm}

Thus in the situation of Chevalley-Shephard-Todd, we have:
$$
QCoh(V//\Gamma)\cong Rep_V(V\times_{V//\Gamma}V)
$$
by faithfully flat descent. Also, since $\cO(V\times_{V//\Gamma}V)$ is free of finite rank over either copy of $\cO(V)$, it is in particular torsion-free and so the natural map
$$
\cO(V\times_{V//\Gamma}V)\to\cO(\Gamma\# V),
$$
which is an isomorphism generically over $V$ for any finite $\Gamma$ (not necessarily generated by reflections), is injective. We conclude that $V\times_{V//\Gamma}V$ is the image of the natural map $\Gamma\# V\to V\times V$, i.e. the adjacency groupoid of $\Gamma\# V$. These conclusions hold also for the action of $\Gamma^x$ on $\liet^*$, since it is conjugate to the action of a finite reflection group, under the automorphism of $\liet^*$ given by translating by $-x$. We will write $\cG^x$ for $\Gamma^x\#\liet^*$ and $\cT^x$, $\cI^x$ for the resulting adjacency groupoid and isotropy subgroup.

In fact, we have the following
\begin{lem}The groupoids $\cG^x$ (over $\liet^*$) and $\cG^x_{(y,y)}$ (over $\liet^*_y$) for any closed point $y$ of $\liet^*$ all satisfy the conditions of Proposition \ref{bigprop}.\end{lem}
\begin{proof} Note that:
$$
\begin{matrix}\cG^x_{(y,y)}&=&(\Gamma^x\cap\Gamma^y\# \liet^*)_{(y,y)}\\
&=&(\Gamma^x\cap\Gamma^y\# \liet^*)_{(y,\liet^*)}\\
&=&(\Gamma^x\cap\Gamma^y)\#\liet^*_y.\end{matrix}
$$
Since $\Gamma^x$, $\Gamma^y$ are both parabolic, so is $\Gamma^x\cap\Gamma^y$; in particular it is generated by reflections. Therefore the adjacency groupoid of $\Gamma^x\cap\Gamma^y\# \liet^*$ is finite flat over $\liet^*$ (with respect to either heads or tails); it is easy to see\footnote{Indeed let $f$ be a function on $\liet^*\times\liet^*$ which does not vanish at $(y,y)$. Consider the function $f^y:=\prod_{\gamma\in\Gamma^x\cap\Gamma^y}\gamma(f)$, where $\gamma$ acts on the second factor of $\liet^*$. This function also does not vanish at $(y,y)$, and its restriction to the adjacency groupoid $A$ in question coincides with the function 
$$
A\to \liet^*\times\liet^*\xrightarrow{\pi_1}\liet^*\xrightarrow{diag}\liet^*\times\liet^*\xrightarrow{f^y}\bA^1.
$$
Thus to invert $f$ on $A$ it suffices to invert some function which factors through $\pi_1$ as claimed.} that the stalk at $(y,y)$ of this adjacency groupoid is the same as its stalk at $(y,\liet^*)$. It follows that this stalk is finite flat over $\liet^*_y$, and also that it coincides with the adjacency groupoid of $\cG^x_{(y,y)}$; this gives condition (1)(a). We note that the adjacency groupoid of $\cG^x_{(y,y)}$ also coincides with $\cT^x_{(y,y)}$. Conditions (1)(c) and (2) are solved in the same way as in points (3) and (4) of paragraph \ref{adj}. This leaves condition (1)(b). Noting that the formation of $\cT^x$, $\cI^x$ from $\cG^x$ commutes with taking stalks, this amounts to checking that 
$$
\cI^x\times\cG^x\rightrightarrows\cG^x\to\cT^x
$$
and
$$
\cI^x_{(y,y)}\times\cG^x_{(y,y)}\rightrightarrows\cG^x_{(y,y)}\to\cT^x_{(y,y)}
$$
are coequalizer diagrams.

For the second diagram, let $z$ be any closed point of $\liet^*$ whose stabilizer in $W^{aff}$ is $\Gamma^x\cap\Gamma^y$. Then, translating by $z-y$, we see that the groupoid $\cG^x_{(y,y)}$ over $\liet^*_y$ is isomorphic to the groupoid $\cG^z_{(z,z)}$ over $\liet^*_z$. This is isomorphic to $\cG_{(z,z)}$ (over $\liet^*_z$). Noting now that the formation of $\cT$, $\cI$ from $\cG$ commutes with taking stalks, we see that the second diagram is isomorphic to the localization at $(z,z)$ of the big (admissible) coequalizer diagram
$$
\cI\times\cG\rightrightarrows\cG\to\cT
$$
and so is indeed a coequalizer diagram.

Consider now the first coequalizer diagram-to-be. To check that it is a coequalizer diagram it suffices to check at stalks of closed points. At any closed point $(y,z)$ say, we need to show that
$$
\cI^x_{(y,y)}\times\cG^x_{(y,z)}\rightrightarrows\cG^x_{(y,z)}\to\cT^x_{(y,z)}
$$
is a coequalizer diagram. If $y$, $z$ are not conjugate under $\Gamma^x$, this is vacuous; otherwise suppose $\gamma\in\Gamma^x$ with $\gamma(z)=y$. Then this diagram is isomorphic to
$$
\cI^x_{(y,y)}\times\cG^x_{(y,y)}\times\Gamma_\gamma\rightrightarrows\cG^x_{(y,y)}\times\Gamma_\gamma\to\cT^x_{(y,y)}\times\Gamma_\gamma
$$
under the multiplication map, which is isomorphic to
$$
\cI^x_{(y,y)}\times\cG^x_{(y,y)}\rightrightarrows\cG^x_{(y,y)}\to\cT^x_{(y,y)}
$$
under the projection, which we have just shown to be a coequalizer diagram.
\end{proof}

Consequently, an $\cO(\cG^x)$-comodule (over $\liet^*$) has at most one compatible structure of $\cO(\cT^x)$-comodule. Likewise, an $\cO(\cG^x)_{(y,y)}$-comodule (over $\liet^*_y$) has at most one compatible structure of $\cO(\cT^x)_{(y,y)}$-comodule. We are now ready to prove\footnote{and indeed, to formulate: in the statement of the theorem the uniqueness of a compatible descent datum is implicit.} the following:
\begin{thm}An object $V$ of $QCoh^{W^{aff}}(\liet^*)$ has trivial derived isotropy if and only if for every finite parabolic subgroup $\Gamma\subset W^{aff}$, the $\Gamma$-equivariant structure on $V$ is descent datum for $\liet^*\to\liet^*//\Gamma$.\end{thm}

\begin{proof}That trivial derived isotropy implies descent for all finite parabolic subgroups is immediate. Conversely, assume the $W^{aff}$-equivariant quasicoherent sheaf $V$ has descent for all finite parabolic subgroups. It means that for each closed point $x\in\liet^*$ there is a unique dashed comultiplication making the diagram
$$
\begin{matrix}V&\to&V\otimes\cO(\cG)\\
\dashdownarrow&&\downarrow\\
V\otimes\cO(\cT^x)&\to& V\otimes\cO(\cG^x)\end{matrix}
$$
commutative. Also for any closed point $y$ the composition $V\to V\otimes\cO(\cT^x)\to V\otimes\cO(\cT^x_{(y,y)})$ is the unique comodule map making the diagram
$$
\begin{matrix}V&\to&V\otimes\cO(\cG)\\
\dashdownarrow&&\downarrow\\
V\otimes\cO(\cT^x_{(y,y)})&\to& V\otimes\cO(\cG^x_{(y,y)})\end{matrix}
$$
commutative. Denote by $\cT_{\liet^*}$ the stalk of $\cT$ at the diagonal $\liet^*\subset\liet^*\times\liet^*$. Choose an enumeration $\gamma_1,\gamma_2,\ldots$ of $W^{aff}$. Set $S_i$ to be the closed subscheme of $\cT$ which is the union of the graphs $\Gamma_{\gamma_1},\ldots,\Gamma_{\gamma_i}$. These exhaust $\cT$. Write also $\Omega_i=\{\gamma_1,\ldots,\gamma_i\}$. For any closed point $(x,y)$ in $\liet^*\times\liet^*$ let us write $\cT^{y\to x}$ for the union of graphs passing through $(x,y)$. This is a torsor for $\cT^x$ in the sense that choosing any component of $\cT^{y\to x}$ gives an isomorphism with $\cT^x$; likewise it is a torsor for $\cT^y$. Similarly write $\Gamma^{y\to x}$ for subset of $W^{aff}$ consisting of all $\gamma$ such that $\gamma(y)=x$, a torsor for both $\Gamma^x$ and $\Gamma^y$. We construct the map $V\to V\otimes\cO(\cT)$ as follows. For each subscheme $S_i$ we form the open cover
$$
\{U^{x,y}_i=S_i - \bigcup_{\gamma\in\Omega_i-\Gamma^{y\to x}}\Gamma_\gamma\}_{(x,y)}.
$$
Each $U^{x,y}_i$ is an open subscheme of a closed subscheme of $\cT^{y\to x}$, and so we define the map
$$
V\to V\otimes \cO(\cT^{y\to x})\to V\otimes\cO(U^{x,y}_i).
$$
Here the map $V\to V\otimes \cO(\cT^{y\to x})$ by definition equals either $0$ (if $\cT^{y\to x}$ is empty) or otherwise the composition
$$
V\to V\otimes\cO(\cT^y)\to V\otimes\cO(\Gamma_\gamma)\otimes\cO(\cT^y)\xrightarrow{\sim} V\otimes \cO(\cT^{y\to x})
$$
for any choice $\gamma\in \Gamma^{y\to x}$ \footnote{Here the first morphism is the comultiplication given by the hypothesis of finite parabolic descent; the second morphism is induced by the morphism $V\to V\otimes\cO(\Gamma_\gamma)$ determined by the $W^{aff}$-equivariant structure; the third morphism is the torsor isomorphism determined by $\gamma$. The composition is independent of $\gamma$.}. To see that these glue together, it suffices to check that for every three closed points $(x_1,x_2),(y_1,y_2),(z_1,z_2)$ of $\liet^*\times\liet^*$ such that $(z_1,z_2)\in U^{x_2\to x_1}_i\cap U^{y_2\to y_1}_i$, the two resulting maps
$$
V\to V\otimes \cO(\cT^{x_2\to x_1})\to V\otimes\cO(U^{x_1,x_2}_i)\to V\otimes\cO((S_i)_{(z_1,z_2)})
$$
and
$$
V\to V\otimes \cO(\cT^{y_2\to y_1})\to V\otimes\cO(U^{y_1,y_2}_i)\to V\otimes\cO((S_i)_{(z_1,z_2)})
$$
coincide. Now $(z_1,z_2)\in U^{x_2\to x_1}_i\cap U^{y_2\to y_1}_i$ implies $\Gamma^{z_2\to z_1}\cap\Omega_i\subset \Gamma^{x_2\to x_1}\cap\Gamma^{y_2\to y_1}$. It follows that our two morphisms can be written as
$$
V\to V\otimes \cO(\cT^{x_2\to x_1})\to V\otimes\cO( \cT^{x_2\to x_1,y_2\to y_1,z_2\to z_1})\to V\otimes\cO((S_i)_{(z_1,z_2)})
$$
and
$$
V\to V\otimes \cO(\cT^{y_2\to y_1})\to V\otimes\cO( \cT^{x_2\to x_1,y_2\to y_1,z_2\to z_1})\to V\otimes\cO((S_i)_{(z_1,z_2)})
$$
where $\cT^{x_2\to x_1,y_2\to y_1,z_2\to z_1}$ denotes the union of those graphs which pass through all three points $(x_1,x_2),(y_1,y_2),(z_1,z_2)$. This is a torsor for the adjacency group $\cT^{x_2,y_2,z_2}$ of the reflection group $\Gamma^{x_2}\cap\Gamma^{y_2}\cap\Gamma^{z_2}$. These morphisms coincide, since either both are $0$ (if $\cT^{x_2\to x_1,y_2\to y_1,z_2\to z_1}$ is empty) or in both cases the induced morphism $V\to V\otimes\cO( \cT^{x_2,y_2,z_2})$ is the unique comodule structure which restricts to the given $\cO(\cG^{x_2,y_2,z_2})$-comodule structure.

We have constructed the maps $V\to V\otimes\cO(S_i)$, which have the property that each composition $V\to V\otimes\cO(S_i)\to V\otimes\cO((S_i)_{(x,y)})$ factors as $V\to V\otimes\cO(\cT^{y\to x})\to V\otimes\cO((S_i)_{(x,y)})$, and it follows that they are compatible as $i$ ranges to $\infty$, so that we get a morphism $V\to V\otimes\cO(\cT)$. That this is a comodule structure can be checked on stalks, where it holds by construction.\end{proof}

\subsection{}We find it interesting to note that we may view $Toda_1(G)$-mod as being made up of the various $QCoh(\liet^*//\Gamma)$, glued together along their common ramified cover $\liet^*$. We do not yet know what to make of this.

\


\begin{thebibliography}{100}

\bibitem{kost}  B. Kostant, Quantization and representation theory, in Representation theory of Lie groups, London
Math. Soc. Lect. Note Series 34 (1979), 287-316.\
\bibitem{Groups}  M.Demazure, P.Gabriel, Groupes algebriques. Tome I: G\'{e}om\'{e}trie algebrique, g\'{e}n\'{e}ralit\'{e}s,
groupes commutatifs, Masson \& Cie, Editeur, Paris; North-Holland Publishing Co., Amsterdam (1970).\
\bibitem{BF} R. Bezrukavnikov, M. Finkelberg, Equivariant Satake category and Kostant-Whittaker reduction, arXiv:0707.3799 (2007).\
\bibitem{GKM} M. Goresky, R. Kottwitz, and R. MacPherson, Equivariant cohomology, Koszul duality,
and the localization theorem, Invent. Math. 131 (1998).\
\bibitem{Ginzburg}  V. Ginzburg, Perverse sheaves on a Loop group and Langlands' duality, arXiv:alg-geom/9511007 (1995)\
\bibitem{Lu} G. Lusztig, Singularities, character formulas, and a q-analogue for weight multiplicities,
in Analyse et topologie sur les espaces singuliers, Ast\'{e}risque 101-102 (1982) 208-229.\
\bibitem{humph} James E. Humphreys. Reflection groups and Coxeter groups, volume 29 of Cambridge Studies
in Advanced Mathematics. Cambridge University Press, Cambridge, 1990.\
\bibitem{C} C. Chevalley, Invariants of finite groups generated by reflections. Amer. J. Math. 77
(1955), 778-782.
\bibitem{ST} G. C. Shephard and J. A. Todd, Finite unitary reflection groups. Canad. J. Math. 6
(1954), 274-304.\
\end{thebibliography}
\end{document}